\documentclass[11pt]{article}
\usepackage{pslatex}
\usepackage{fancyhdr}
\usepackage{graphicx}
\usepackage{geometry}
\RequirePackage[latin1]{inputenc} \RequirePackage[T1]{fontenc}

\def\figurename{Figure} 
\makeatletter
\renewcommand{\fnum@figure}[1]{\figurename~\thefigure.}
\makeatother

\def\tablename{Table} 
\makeatletter
\renewcommand{\fnum@table}[1]{\tablename~\thetable.}
\makeatother
\usepackage{color}
                [2000/03/29 v1.0m Standard LaTeX package]

\usepackage{amsmath}
\usepackage{amssymb}
\usepackage{amsfonts}
\usepackage{amsthm,amscd}
\newtheorem{theorem}{Theorem}[section]

\newtheorem{corollary}[theorem]{Corollary}
\newtheorem{proposition}[theorem]{Proposition}
\theoremstyle{definition}

\newtheorem{definition}[theorem]{Definition}
\newtheorem{example}[theorem]{Example}

\theoremstyle{remark}
\newtheorem{remark}[theorem]{Remark}

\numberwithin{equation}{section}

\def\P{\mathbb P}
\def\R{\mathbb R}
\def\E{\mathbb E}

\def\E{\mathbb E}
\def\N{\mathbb N}

\begin{document}
\title{\textbf{Probabilistic representation of the parabolic stochastic variational inequality with Dirichlet-Neumann boundary and variational generalized backward doubly stochastic differential equations}\footnote{The work of Yong Ren is supported by the National
Natural Science Foundation of
 China (No 11371029),
 the Distinguished Young Scholars of Anhui Province (No 1508085JGD10). The work of Qing Zhou is supported by the National Natural
Science Foundation of China (No 11001029). The work of Auguste Aman is supported by TWAS Research Grants to individuals (No. 09-100 RG/MATHS/AF/\mbox{AC-I}--UNESCO FR: 3240230311)}
}

\author{\textbf{ Yong Ren}$^1$\footnote{e-mail: renyong@126.com},\;\;\ \
\textbf{ Auguste Aman}$^2$\footnote{Corresponding author e-mail: augusteaman5@yahoo.fr}\;\,, \ \ \textbf{ Qing Zhou}$^3$ \footnote{ e-mail: zqleii@gmail.com}\\
\noindent\small 1. Department of Mathematics, Anhui Normal University, Wuhu
241000, China\;\;\;\;\;\;\;\;\;\;\;\;\;\;\;\;\;\;\;\;\;\;\;\;\;\;\;\\
\noindent\small 2. U.F.R Mathematiques et Informatique, Universit\'{e} Félix H. Boigny de
Cocody, C\^{o}te d'Ivoire\;\;\;\;\; \\
\noindent\small 3. School of Science, Beijing University of Posts and Telecommunications, Beijing 100876, China} \;
\date{}
\maketitle

\begin{abstract}
We derive the existence and uniqueness of the generalized backward doubly stochastic differential equation with subdifferential of a lower semicontinuous convex function under a non-Lipschitz condition. This study allows us to give a probabilistic representation (in the stochastic viscosity sense) to the parabolic  variational stochastic partial differential equations with Dirichlet-Neumann conditions.
\end{abstract}

\noindent {\bf 2000 MR Subject Classification} 60H15, 60H20

\vspace{.08in} \noindent \textbf{Keywords} Backward doubly
stochastic differential equation, subdifferential operator,
Yosida approximation, Neumann-Dirichlet boundary conditions,
stochastic viscosity solution

\section{Introduction}
Since the pioneering work of Pardoux and Peng in \cite{PP}, many others work on the connection between backward doubly stochastic differential equations (BDSDEs, for short) and semi linear parabolic and elliptic systems of second-order stochastic partial differential equations (SPDEs, for short) have been established. Among these, the notion of stochastic viscosity solution for semi linear SPDEs has been introduced firstly by Lions and Souganidis in \cite{LS1,LS2}. Their idea is to remove the stochastic integral from a SPDEs using a so-called "stochastic characteristic". Years later, two others definitions of stochastic viscosity solutions of SPDEs have been considered, respectively, in \cite{BM1,BM2,BMa} using the "Doss-Sussman" transformation to  remove the stochastic integrals from a SPDEs. Using this similar approach, Boufoussi et al. \cite{Bal} derive the stochastic viscosity solution of SPDEs with nonlinear Dirichlet-Neumann boundary condition via a so-called generalized BDSDEs.

Let $\Theta$ be an open connected and smooth bounded subset of $\R^d$. More precisely, we suppose that for a function $\phi\in C^2((\R^d),\Theta$ and its boundary $Bd(\Theta)$ are characterized by 
\begin{eqnarray*}
\Theta=\{x\in \R^d:\;\;\phi(x)>0\},
\end{eqnarray*}
\begin{eqnarray*}
Bd(\Theta)=\{x\in\R^d:\;\;\phi(x)=0\}.
\end{eqnarray*}
and, for any $x\in Bd(\Theta), \nabla\phi(x)$ is the unit normal vector pointing towards the interior of $\Theta$.
  
Given continuous mappings: $\sigma:\R^d\to \R^{d\times d},\; b: \R^d \to \R^{d}$, we consider reflected stochastic differential equations (RSDEs, in short)
\begin{eqnarray}\label{equation011}
X_s^{t,x}&=&x+\int^s_t b(X^{t,x}_r)dr +\int_t^s \nabla \phi(X^{t,x}_r)dA_r^{t,x}+\int_t^s\sigma(X^{t,x}_r)dW_r
\end{eqnarray}
and its associated second-order differential operators
 \begin{eqnarray}\label{Dop}
L=\sum_{i,j=1}^{d}(\sigma\sigma*)_{ij}(x)\frac{\partial^2 }{\partial x_i\partial x_j}+\sum_{i=1}^{d}b_i(x)\frac{\partial}{\partial x_i},\; x\in \Theta
\end{eqnarray}
and 
\begin{eqnarray}\label{Dop1}
\Gamma =\sum_{i=}^{d}\frac{\partial \phi}{\partial x_i}(x)\frac{\partial }{\partial x_i}, \; x\in Bd(\Theta)
\end{eqnarray} 

Next, for some function $f, g:\Omega\times[0,T]\times\overline{\Theta}\times\R \to \R,\; \chi: \R^d\to  \R$, we consider the following second-order semilinear parabolic SPDE
\begin{eqnarray}
\left\{
\begin{array}{lll}
\frac{\partial u}{\partial t}(t,x)+Lu(t,x)+f(t,x,u(t,x))+h(t,x,u(t,x))\frac{\overleftarrow{dB}_t}{\partial t}=0, 
(t,x)\in [0,T]\times \Theta,\\\\
\Gamma u(t,x)+g(t,x,u(t,x))=0, \
(t,x)\in [0,T]\times Bd(\Theta),\\\\
u(T,x)=\chi(x),\;\;x\in\overline{\Theta}.
\end{array}\right.,
\label{i2}
\end{eqnarray}
where a $\frac{\overleftarrow{dB}_t}{\partial t}$ denotes the "white noise" defined formally as the derivative of the Brownian motion $B$.

In the case $h\equiv 0$ and the functions $f, g$ are deterministic, SPDE \eqref{i2} becomes a well-known PDE studied by Pardoux and Zhang \cite{PZ}.

Furthermore,  Boufoussi et al. in \cite{Bal} prove that $u$, defined by
\begin{eqnarray*}
	u(t,x)=Y^{t,x}_t,\; \mbox{for all} \; (t,x)\in[0,T]\times\Theta,
\end{eqnarray*}
where $\{Y^{t,x}_s, \; s\in [t,T]\}$ satisfied the following generalized backward doubly stochastic differential equation (GBDSDE, for short)
\begin{eqnarray}\label{equation01}
Y_s^{t,x}&=&\chi(X^{t,x}_T)+\int^T_sf(r,X^{t,x}_r,Y^{t,x}_r,Z^{t,x}_r)dr +\int_s^Tg(r,X^{t,x}_r,Y^{t,x}_r)dA_r^{t,x}\nonumber\\
&&+\int_s^T h(r,X_r^{t,x},Y^{t,x}_r,Z^{t,x}_r)\overleftarrow{dB}_r-\int_s^TZ^{t,x}_rdW_r,
\end{eqnarray}
is the stochastic viscosity solution of \eqref{i2}.

In certain applications, the solution of \eqref{i2} need to be maintained in non-empty closed convex real sets $D_1$ (resp. $D_2$) for all $x\in \Theta$ (resp. $x\in Bd(\Theta)$. This constraint requires the addition of sources $\partial{\mathbb I}_{D_1}(u(t,x))$ and $\partial{\mathbb I}_{D_2}(u(t,x))$ which design respectively the sub-differential of the convex indicators functions ${\mathbb I }_{D_1}$ and ${\mathbb I}_{D_2}$ defined by
\begin{eqnarray*}
{\mathbb I }_{D_i}(y)=
\left\{
\begin{array}{lll}
0& \mbox{if}& y\in D_i\\
+\infty &\mbox{if}& y\notin D_i,\;\; i=1,2
\end{array}\right.
\end{eqnarray*}
These sources inward pushes process $u(t,x)$ in $D_1$, for all $x\in\Theta$ and in $D_2$, for all $x\in Bd(\Theta)$ in a minimal way (i.e. only when $u(t, x)$ arrives respectively on the boundary of $D_1$ and $D_2$) such that $\partial{\mathbb I }_{D_1}(u(t,x))$ and $\partial{\mathbb I }_{D_2}(u(t,x))$ represent perfect feedback flux controls.

In this paper, we study an existence result for this parabolic variational inequality with mixed nonlinear Neumann-Dirichlet boundary condition:

\begin{eqnarray}
\left\{
\begin{array}{lll}
\Big(\frac{\partial u}{\partial t}(t,x)+Lu(t,x)+f(t,x,u(t,x))+h(t,x,u(t,x))\frac{\overleftarrow{dB}_t}{\partial t}\Big)\in\partial \varphi(u(t,x)), 
(t,x)\in [0,T]\times \Theta,\\\\
\Gamma u(t,x)+g(t,x,u(t,x))\in \partial \psi(u(t,x)), \
(t,x)\in [0,T]\times Bd(\Theta),\\\\
u(T,x)=\chi(x),\;\;x\in\overline{\Theta},
\end{array}\right.
\label{i1}
\end{eqnarray}
where $\partial\varphi$ and $\partial\psi$ are sub-differential operators of the convex function $\varphi$ and $\psi$.

The special case (when $h\equiv 0$) of SPDE \eqref{i1} corresponds to now well-know variational PDE with Dirichlet-Neumann boundary condition studied in \cite{MR}.

Our method is purely probabilistic and use a Markovian version of this following variational generalized backward doubly stochastic differential equations (VGBDSDEs, in short)
\begin{eqnarray}\label{equation1}
\left\{
\begin{array}{lll}
dY_t+f(t,Y_t,Z_t)dt+g(t,Y_t)dA_t+h(t,Y_t,Z_t)\overleftarrow{dB}_t\in \partial
\varphi(Y_t)dt+\partial \psi(Y_t)dA_t+Z_tdW_t&\\\\
Y_T=\xi, &
\end{array}\right.
\end{eqnarray}
where $\xi$ is the terminal value and $A$ denotes a continuous one-dimensional increasing adapted process such that $A_0=0$.

This study is done under "{\bf Main assumption}" introduced by X. Mao in \cite{XM} and defined as follows:
 \begin{description}
 \item
 There exist two constants $K>0$ and $0<\alpha <1$
such that
\begin{equation*}
\left\{
\begin{array}{rll}
\mid f(t,y_{1},z_{1})-f(t,y_{2},z_{2})\mid ^{2}\leq \rho (\mid
y_{1}-y_{2}\mid ^{2})+K\mid \mid z_{1}-z_{2}\mid \mid ^{2}\vspace{0.2cm} &
\\\\
\mid \mid h(t,y_{1},z_{1})-h(t,y_{2},z_{2})\mid \mid ^{2}\leq \rho (\mid
y_{1}-y_{2}\mid ^{2})+\alpha \mid \mid z_{1}-z_{2}\mid \mid ^{2},&
\end{array}%
\right.
\end{equation*}
where $\rho (.)$ is continuous, concave and non-decreasing function from $\R^{+}$ to $\R^{+}$ such that $\rho(0)=0$, for all $u>0,\; \rho(u)>0$ and 
\begin{eqnarray*}
\int_{0^{+}}\frac{du}{\rho(u)}=+\infty.
\end{eqnarray*}
\end{description}
Several kinds of BSDEs have been study under this assumption and it relaxed version (see Owo and N'zi \cite{ON}, Wand and Wang \cite{WW}, Wang and Huang \cite{WH}).

Furthermore, it is now well-know that the continuity of the map $(t,x)\mapsto Y^{t,x}$ plays an important role for viscosity solution of SPDE \eqref{i1}. In this paper, we prove this result in Proposition 3.11. It can be view as a stochastic version of continuity result established by Pardoux and  R\u{a}\c{s}canu in \cite{PR1}. Since the solution $u$ of \eqref{i1} is a stochastic fields, we use weak convergence to derive this continuity result.  

The rest of this paper is organized as follows. In Section 2, we study VGBDSDEs \eqref{equation1} under "{\bf main assumptions}". The last section derives and establishes stochastic viscosity solution for a one dimensional variation stochastic Dirichlet-Neumann problems in the special case (Lipschitz coefficients).

\section{Variational Generalized Backward Doubly SDEs}
This section aims to study VGBDSDEs \eqref{equation1} under above assumptions. In subsection 2.1, we give the statement of our study. We derive existence and uniqueness result in subsection 2.2 and subsection 2.3 is devoted to comparison theorem.

\subsection{ Preliminaries and notations}
For a final time $T>0$, we consider $\{W_{t}; 0\leq t\leq
T \}$ and $\{B_{t}; 0\leq t\leq T \}$ two standard Brownian motion defined respectively on complete probability spaces $(\Omega_1,\mathcal{F}_1,\P_1)$ and $(\Omega_2,
\mathcal{F}_2,\P_2)$ with respectively $\mathbb{R}^{d}$ and $\mathbb{R}^{l}$ values. For any process $\{K_{t},\, t\in[0,T]\}$, we define the following family of $\sigma$-algebra $\mathcal{F}_{s,t}^{K}=\sigma\{K_{r}-K_{s},\, s\leq r \leq t \}$. In particular, $\mathcal{F}_{t}^{K}=\mathcal{F}_{0,t}^{K}$. Next, we consider the product space $(\Omega, \mathcal{F},\P)$, where
$$\Omega=\Omega_{1}\times\Omega_{2},\ \mathcal{F}=\mathcal{F}_{1}\otimes\mathcal{F}_{2}, \ \P=\P_{1}\otimes\P_{2}$$
and $\mathcal{F}_{t}=\mathcal{F}_{t}^{W} \otimes \mathcal{F}_{t,T}^{B}$. We should note that since $\textbf{\textit{F}}^{W}=(\mathcal{F}_{t}^W)_{t\in [0,T]}$ and $\textbf{\textit{F}}_{}^{B}=(\mathcal{F}_{t,T}^B)_{t\in [0,T]}$ are respectively increasing and decreasing filtration, the collection $\textbf{\textit{F}}_{}^{}=(\mathcal{F}_{t})_{t\in[0,T]}$ is neither increasing nor decreasing such that it is not a filtration. Further, for random variables $\zeta$ and $\pi$ defined respectively in $\Omega_1$ and $\Omega_2$ are viewed as random variables on $\Omega$ via the following identification:
$$
\zeta(\omega)=\zeta(\omega_1);\;\;\;  \pi(\omega)=\pi(\omega_2), \;\;  \omega=(\omega_1,\omega_2).
$$
Let
$\mathcal{M}^{2}(\textbf{\textit{F}},[0,T];\mathbb{R}^{d\times k})$ denote the set of $d\P\otimes dt$ a.e. equal and $(d\times k)$-dimensional jointly measurable
random processes  $\{\varphi_{t}; 0\leq t\leq T \}$  such that
\begin{enumerate}
\item[(i)] $\displaystyle \|\varphi \|_{\mathcal{M}^{2}}^{2}=\mathbb{E}\left(\int_{0}^{T}\| \varphi_{t}
\|^{2} dt\right)< +\infty$,

\item[(ii)] $\varphi_{t}$ is $\mathcal{F}_{t}$-measurable, for
a.e. $t \in [0,T].$
\end{enumerate}
We denote by  $\mathcal{S}^{2}(\textbf{\textit{F}},[0,T];\mathbb{R}^k)$
the set of continuous $k$-dimensional random processes such that
\begin{enumerate}
\item[(i)] $\|\varphi \|_{\mathcal{S}^{2}}^{2}=\mathbb{E}(\underset{0\leq t\leq T}
{\sup} \mid \varphi_{t}\mid^{2})< +\infty$,

\item[(ii)] $\varphi_{t}$ is $\mathcal{F}_{t}$-measurable, for any
$t \in [0,T]$.
\end{enumerate}
We denote also by $\mathcal{A}^{2}(\textbf{\textit{F}},[0,T];\mathbb{R}^k)$ the set of $d\P\otimes dA_t$ a.e. equal and $k$-dimensional jointly measurable random processes  $\{\varphi_{t}; 0\leq t\leq T \}$ such that
\begin{enumerate}
\item[(i)] $\displaystyle \|\varphi \|_{\mathcal{A}^{2}}^{2}=\mathbb{E}\left(\int_{0}^{T}| \varphi_{t}
|^{2} dA_t\right)< +\infty$,

\item[(ii)] $\varphi_{t}$ is $\mathcal{F}_{t}$-measurable, for
a.e. $t \in [0,T].$
\end{enumerate}
In the sequel, for simplicity, we shall set $\mathcal{S}^{2}(\R^k),\,\mathcal{M}^{2}(\R^{d\times k})$ and $\mathcal{A}^{2}(\R^k)$ instead of \newline$\mathcal{S}^{2}(\bf{F},[0,T];\R^k)$, $\mathcal{M}^{2}({\bf F},[0,T],\R^{d\times k})$ and $\mathcal{A}^{2}({\bf F},[0,T],\R^k)$ respectively.
 
The coefficients $f:\Omega \times [0,T]\times \mathbb{R}^k \times \mathbb{R}^{
d\times k}\rightarrow \mathbb{R}^k$, $h:\Omega \times [0,T]\times \mathbb{R}^k
\times \mathbb{R}^{ d\times k}\rightarrow \mathbb{R}^{l\times k}$, $g:\Omega \times [0,T]\times \mathbb{R}^k\rightarrow \mathbb{R}^k$, the terminal value $\xi: \Omega \rightarrow \mathbb{R}^k $ and  $\varphi, \psi:\mathbb{R}^k\rightarrow \R$ satisfy the following conditions:
\begin{description}
\item [(H1)]$\xi$ is a $\mathcal{F}_T$-measurable random variable such that $\E(e^{\mu A_T}|\xi|^2+\varphi(\xi)+\psi(\xi))<+\infty$, for any $\mu>0$
\item [(H2)] There exist three constants $K> 0,\; 0 <\alpha < 1$ and $\beta\in\R$ such that   
\begin{eqnarray*}
\left\{
\begin{array}{lll}
(i)&  f(.,0,0) \in\mathcal{M}^{2}(\R^k), \ \ h(.,0,0) \in
\mathcal{M}^{2}(\R^{l\times k}),\; g(s,0)\in \mathcal{A}^{2}(\R^k),&\\\\ 
(ii)&\mid f(t,y_{1},z_{1})-f(t,y_{2},z_{2}) \mid^{2} \leq \rho(\mid
y_{1}-y_{2}\mid^{2})+K\mid \mid z_{1}-z_{2}\mid \mid^{2},&
\\\\
(iii)&\mid \mid h(t,y_{1},z_{1})-h(t,y_{2},z_{2}) \mid \mid^{2} \leq \rho(\mid
y_{1}-y_{2}\mid^{2})+\alpha \mid \mid z_{1}-z_{2}\mid \mid^{2},& \\\\
(iv)&\langle y_1-y_2,g(t,y_1)-g(t,y_2)\rangle\leq \beta |y_1-y_2|^{2},&
\end{array}
\right.,
\end{eqnarray*}
for all $(t, y_{i},z_{i})\in [0,T]\times\mathbb{R}
\times \mathbb{R}^{d}$ $i=1,2$, and $\rho$ continuous, concave and non-decreasing function from $\R^{+}$ to $\R^{+}$ satisfies $\rho(0)=0,\; \rho(u)>0$ for all $u>0$ and 
\begin{eqnarray}
\int_{0^{+}}\frac{du}{\rho(u)}=+\infty.\label{A}
\end{eqnarray}

\item[(H3)] There exist constants $K',\mu >0$ and function $\gamma,\,\eta:[0,+\infty[\times\Omega\rightarrow[0,+\infty[$
 \begin{eqnarray*}
\left\{
\begin{array}{lll}
(i)&\mid f(t,y,z)\mid\leq \gamma_t+K'(|y|+\|z\|),&
\\\\
(ii)& | g(t,y)|\leq \eta_t+K'|y|,&
\\\\
(iii)& \displaystyle \E\left(\int^T_0e^{\mu A_T}|\gamma_t|^2dt + \int^T_0 e^{\mu A_T}|\eta_t|^2dA_t\right),
\end{array}
\right.
\end{eqnarray*}

for all $(t, y,z)\in [0,T]\times\mathbb{R}\times \mathbb{R}^{d}$.
\item [(H4)] The functions $\varphi$ and $\psi$ are proper convex and lower semi-continuous  such that 
$\varphi(y)\geq \varphi(0)=0,\ \psi(y)\geq \psi(0)=0 $.
\end{description}

\begin{remark}
\begin{itemize}
\item[(i)] Lipschitz condition on generators $f,\, h$ with respect to the variable $y$ is the special case of $(\bf H2)$ with $\rho(u)=Ku$. However, "{\bf Main assumption}" can be elaborated with other functions outside the Lipschitz case. For example $\rho_1$ and $\rho_2$ defined respectively by
\begin{eqnarray*}
\rho_1(u)=\left\{
\begin{array}{ll}
u\log(u^{-1}),& 0\leq u\leq \delta,\\\\
\delta\log(\delta^{-1})+\kappa_1(\delta)(u-\delta),& u>\delta
\end{array}\right. 
\end{eqnarray*}
and 
\begin{eqnarray*}
\rho_2(u)=\left\{
\begin{array}{ll}
u\log(u^{-1})\log(\log(u)),& 0\leq u\leq \delta,\\\\
\delta\log(\delta^{-1})\log(\log(\delta))+\kappa_2(\delta)(u-\delta),& u>\delta,
\end{array}\right. 
 \end{eqnarray*}
 for any $\delta\in (0,1)$, are perfect candidates.
\item [(ii)] In classical BDSDEs case, i.e $\varphi={\bf 1}_{\R}$ and $g=\psi=0$, it was explained in \cite{PP} (Remark 3.3), why the condition $(\bf H3)$-$(iii)$ with $0<\alpha <1$, is a very natural condition for many particular situations, e.g., if the coefficient $h$ is independent on $y$ and linear with respect to $z$.
\item[(iii)] If $g$ satisfies $(\bf H2 )$-$(iv)$ then, by changing the solutions and the coefficients $f , g$ and $h$, we may and do suppose that $h$ satisfies a stronger condition of the form
\begin{eqnarray}
	 \langle y_1-y_2,g(t,y_1)-g(t,y_2)\rangle\leq \beta |y_1-y_2|^{2}, \; \beta<0.
\end{eqnarray}\label{Neg}
Indeed, $(Y_t,Z_t)$ solves VGBDSDE \eqref{equation1} if and only if for some $\mu>0$ the pair $(\bar{Y}_t, \bar{Z}_t ) = (e^{\mu A_t} Y_t , e^{\mu A_t}Z_t)$ solves an analogous VGBDSDE, with $f , g$ and $h$ replaced respectively by
\begin{eqnarray*}
	\bar{f}(t,y,z)&=& e^{\mu A_t} f(t, e^{-\mu A_t}y,e^{-\mu A_t}z),\\
	\bar{h}(t,y,z)&=& e^{\mu A_t} h(t, e^{-\mu A_t}y,e^{-\mu A_t}z),\\
	\bar{g}(t,y)&=& e^{\mu A_t} g(t, e^{-\mu A_t}y)-\mu y.
\end{eqnarray*}
Then we can always choose $\mu$ such that the function $\bar{g}$ satisfies $(\bf H2 )$-$(iv)$ with a strictly negative $\beta$. 
\end{itemize}
\end{remark}

In order to understand the notion of Yosida approximation of sub-differential operator of functions $\varphi$ (resp. $\psi$), we define
\begin{eqnarray*}
\begin{array}{l}
Dom(\theta)=\left \{u\in \mathbb{R}^k :\theta(u)<+\infty
\right\},\\\\
\partial \theta(u)= \{u^\ast\in
\mathbb{R}^k:\left<u^\ast,v-u\right>+\theta(u)\leq \theta(v),\
\mbox{for all}\ v\in \mathbb{R}^k\},\\\\
{Dom}(\partial \theta)= \{u\in \mathbb{R}^k:\partial \theta(u)
\neq \emptyset\}.
\end{array}
\end{eqnarray*}
$\partial\theta$, called sub-differential operator of $\theta$, is a maximal monotone operator, which means that
\begin{eqnarray*}
\langle u-v,u^*-v^*\rangle\geq 0,\;\; \forall\; (u,u^*),\,(v,v^*)\in\partial\theta.
\end{eqnarray*}
Moreover we have 
\begin{eqnarray*}
	(u,u^*)\in\partial \theta \Leftrightarrow u\in {\rm Dom} (\theta),\; u^\ast \in\partial\theta(u).
\end{eqnarray*}
 Example, for $\theta:\R\to ]-\infty,+\infty]$ and $y\in Dom(\theta)$ we have
\begin{eqnarray*}
\partial\theta(y)=\R\cap [\theta'_{l}(y),\theta'_{r}(y)]	,
\end{eqnarray*}
where $\theta'_{l}(y)$ and $\theta'_{r}(y)$ denote respectively the left and right derivative  at $y$.
\begin{remark}
Assumption $(\bf H4)$ is not really a restriction. It is generally realized by changing problem form. For $u_0\in Dom(\varphi)$ (resp. $u_0\in Dom(\psi)$), if $(u_0, u^{*}_{0})\in\partial\varphi$ (resp. $(u_0, u^{*}_{0})\in\partial\psi$) we can replace $\varphi(u)$ (resp. $\psi(u)$) by $\varphi(u+u_0)-\varphi(u_0)-\langle u^{*}_0, u\rangle$ (resp. $\psi(u+u_0)-\psi(u_0)-\langle u^{*}_0, u\rangle$).
\end{remark}
\begin{definition}[see \cite{Br} and the references therein]
Let consider $(\theta_\varepsilon)_{\varepsilon>0}$, the family of function  defined by \begin{eqnarray*}
\theta_\varepsilon(x)=\min_{y}\left(\frac{1}{2\varepsilon}|x-y|^2+\theta(y)\right)=
\frac{1}{2\varepsilon}\left|x-J_\varepsilon(x)\right|^2+\theta(J_\varepsilon(x)),
\end{eqnarray*}
where $J_{\varepsilon}(x)=(I+\varepsilon\partial \theta)^{-1}(x)$ is called the resolvent of the monotone operator $\partial\theta$. The Yosida approximation of the sub-differential operator $\partial\theta$ is the family of function $(\theta_{\varepsilon})_{\varepsilon>0}$ defined by
\begin{eqnarray*}
\nabla\theta_{\varepsilon}(x)=\frac{x-J_{\varepsilon}(x)}{\varepsilon}.
\end{eqnarray*} 
\end{definition}
The following results come from \cite{Br} or \cite{PR}.
\begin{proposition}\label{pro2.1}
\begin{itemize}
\item[(i)] The function $\theta_\varepsilon$ is convex with a Lipschitz continuous derivatives;
\item[(ii)] for all $ x\in \mathbb{R}^k,
\nabla\theta_\varepsilon(x)=\partial
\theta_\varepsilon(x)=\frac{1}{\varepsilon}\left(x-J_\varepsilon(x)\right)\in
\partial \theta(J_\varepsilon(x))$;
\item[(iii)] for all $ x,y\in \mathbb{R}^k,
|\nabla\theta_\varepsilon(x)-\nabla\theta_\varepsilon(y)|\leq
\frac{1}{\varepsilon}
|x-y|$;
\item[(iv)] for all $ x,y\in \mathbb{R}^k,
\left<\nabla\theta_\varepsilon(x)-\nabla\theta_\varepsilon(y),x-y\right>\geq
0$;
\item[(v)] for all $x,y\in\mathbb{R}^k\;\mbox{and} \;
\varepsilon,\delta>0$,
$\left<\nabla\theta_\varepsilon(x)-\nabla\theta_\delta(y),
x-y\right>\geq-\left(\varepsilon+\delta\right)\left<\nabla\theta_\varepsilon(x),\nabla\theta_\delta(y)\right>$.
\end{itemize}
\end{proposition}
Let us now give the definition of an adapted solution to our VGBDSDEs. 
\begin{definition}\rm \label{def2.1}
The quadruplet of processes $(Y,U,V,Z)$ is called a solution of VGBSDE \eqref{equation1} if
\begin{itemize} 
\item[(i)] $(Y,Z,U,V)\in\mathcal{S}^{2}(\R^k)\times\mathcal{M}^{2}(\R^{d\times k})\times\mathcal{M}^{2}(\R^{k})\times\mathcal{A}^{2}(\R^k)$
\item[(ii)]$(Y_t,U_t)\in \partial \varphi,\;\; d\P\otimes dt$-a.e and $(Y_t,V_t)\in \partial \psi, \;\; d\P\otimes dA_t$;
\item[(iii)]
\begin{eqnarray}\label{Eqdef}
Y_t+\int_t^T U_sds+\int_t^T V_s dA_s &=&\xi+\int_t^Tf(s,Y_s,Z_s)ds+\int_t^Tg(s,Y_s)dA_s\nonumber\\
&&+\int_t^T h(s,Y_s,Z_s)\overleftarrow{dB}_s-\int_t^T Z_s dW_s,\,\;\;\;  0\leq t \leq T.
\end{eqnarray}
\end{itemize}
\end{definition}
To end this subsection, let us give Bihari's inequality which is useful in the proof of our existence and uniqueness result.
\begin{proposition}
Let $u$ and $f$ be non-negative continuous functions defined on $[0, +\infty)$, and let $w$ be a continuous non-decreasing function defined on $[0,+\infty)$ and $w(u) > 0$ on $(0, +\infty)$. If $u$ satisfies the following integral inequality,
\begin{eqnarray*}
u(t)\leq \alpha+\int^t_0f(s)w(u(s))ds
\end{eqnarray*}
where $\alpha$ is a non-negative constant, then
\begin{eqnarray*}
u(t)\leq G^{-1}\left(G(\alpha)+\int^t_0f(s)ds\right),
\end{eqnarray*}
with the function $G$ is defined by
\begin{eqnarray*}
G(x)=\int_{x_0}^{x}\frac{1}{w(y)}dy, \;\;\;\; x>0,\;\; x_0>0.
\end{eqnarray*}
\end{proposition}

\subsection{Existence and uniqueness result}

The real difficulty to derive the existence and uniqueness result for VGBDSDEs \eqref{equation1} resides in the fact that when we use the Yosida approximations, we obtain this following doubly BSDEs \begin{eqnarray}
Y_t=\xi+\int_t^T f(s,Y_s,Z_s)ds +\int_t^T g(s,Y_s)dA_s+\int_t^T h(s,Y_s,Z_s)\overleftarrow{dB}_s-\int_t^T Z_sdW_s.\label{EqGBSDE}
\end{eqnarray}
introduced in \cite{Bal} by Boufoussi et al. Under at least Lipschitz assumption on the coefficients, they establish existence and uniqueness result. Unfortunately, this assumption doesn't work in our context. Therefore before starting our main existence and uniqueness result, we need to establish a general existence and uniqueness result for GBDSDEs \eqref{EqGBSDE} under "{\bf Main assumptions}". 
\begin{theorem}\label{Tfirst}
Assume assumptions $(\bf H1)$-$(\bf H2)$ hold. Then, GBDSDEs \eqref{EqGBSDE} has a unique solution.
\end{theorem}
\begin{remark}
We emphasize that when $g\equiv 0$, equation \eqref{EqGBSDE} becomes the classical backward doubly stochastic equations (BDSDEs, in short) introduce by Pardoux and Peng \cite{PP} under Lipschitz generator and recently relax by many authors. We can see in  \cite{ON}, Owo and N'zi's work and references therein.   
\end{remark}
\begin{proof}
{\bf $(i)$ Existence}\\
To prove existence result,  we follow the well-know idea by
constructing a Picard scheme and show its convergence.
In this fact, let set $Y^0_s =0 $ and define for $n\in \mathbb{N}$ recursively
\begin{eqnarray}
Y^{n}_t &=& \xi + \int_t^T f(s,Y_s^{n-1},Z_s^n)ds+\int_t^Tg(s,Y_s^{n})dA_s+\int_t^T h(s,Y_s^{n-1},Z_s^n)\overleftarrow{dB}_s\nonumber\\
&&- \int_t^T Z^{n}_sdW_s,\; 0 \leq t \leq T.\label{eq5}
\end{eqnarray}
Let us remark that for $Y^{n-1}\in\mathcal{S}^{2}(\R^k)$, generators $f$ and $h$ depend only of $z$ and are Lipschitz. Therefore thanks to Theorem 2.1 in \cite{Bal}, GBDSDEs  \eqref{eq5} admits a unique solution $(Y^{n},Z^{n})\in\mathcal{S}^{2}(\R^k)\times \mathcal{M}^{2}(\R^{d\times k})$. Our purpose is to prove that the sequence of processes $(Y^{n},Z^{n})$ converges in $\mathcal{S}^{2}(\R^k)\times \mathcal{M}^{2}(\R^{d\times k})$ and its limit solves GBDSDEs \eqref{EqGBSDE}. The proof will subdivided into four steps\\
{\it Step 1:} There exist $T_1\in[0,T)$ and $M_1\geq 0$ such that for each $n\geq 1$, 
\begin{eqnarray}\label{S2}
\E\left(|Y^n_t|^2+\int^T_t|Y^n_s|^2dA_s+\int^T_t\|Z^n_s\|^2ds\right)\leq M_1, \;\; \forall\, t\in[T_1,T].
\end{eqnarray}
Applying again the generalized Itô's formula, we obtain
\begin{eqnarray*}
\E(|Y^{n}_t|^2)+\int_t^T\|Z^{n}_s\|^2ds&=&\E(|\xi|^2)+2\E\left[\int_t^T\langle Y^{n}_s,\,f(s,Y^{n-1}_s,Z^{n}_s)\rangle ds\right]\\
&&+2\E\left[\int_t^T\langle Y^n_s,\,g(s,Y^{n}_s)\rangle dA_s\right]\\
&&+2\E\left[\int_t^T\|h(s,Y^{n-1}_s,Z^{n}_s)\|^2ds\right]
\end{eqnarray*}
There exist by Young inequality a constants $\lambda> 0$ such that 
\begin{eqnarray*}
2\langle Y^n_s,\,f(s,Y^{n-1}_s,Z^n_s)\rangle & =& 2\langle Y^n_s,\,f(s,Y^{n-1}_s,Z^n_s)-f(s,0,0)\rangle +2\langle Y^n_s,\,f(s,0,0)\rangle\\
&\leq& \left(\frac{1}{\lambda}+1\right)|Y^n_s|^2+\lambda\rho(|Y^{n-1}_s|^2)+\lambda K\|Z^n_s\|^2+|f(s,0,0)|^2,
\end{eqnarray*}
and 
\begin{eqnarray*}
\|h(s,Y^{n-1}_s,Z^n_s)\|^2& \leq& (1+\lambda)\|h(s,Y^{n-1}_s,Z^n_s)-h(s,0,0)\|^2+(1+\frac{1}{\lambda})\|hs,0,0)|^2\\
&\leq& (1+\lambda)\rho(|Y^{n-1}_s|^2)+\alpha(1+\lambda)\|Z^n_s\|^2+(1+\frac{1}{\lambda})\|h(s,0,0)\|^2.
\end{eqnarray*}
Next, since $\beta$ appear in $(\bf H2)$ can be suppose negative, we have
\begin{eqnarray*}
2\langle Y^n_s,\,g(s,Y^{n}_s) \rangle & =& 2\langle Y^n_s,\,g(s,Y^{n}_s)-g(s,0)\rangle +2\langle Y^n_s,\,g(s,0)\rangle\\
&\leq& 2\beta|Y^n_s|^2+|\beta||Y^n_s|^2+\frac{1}{|\beta|}|g(s,0)|^2\\
&=& -|\beta||Y^n_s|^2+\frac{1}{|\beta|}|g(s,0)|^2.
\end{eqnarray*}
Then, it follows by choosing $0<\lambda<\frac{1-\alpha}{\alpha+K}$ there exist $0<\bar{\alpha}_1=1-(\alpha(1+\lambda)+\lambda K)$ and $\mu=1+\frac{1}{\lambda}$ such that  
\begin{eqnarray}\label{R1}
&&\E\left[|Y^{n}_t|^2+|\beta|\int^T_t|Y^n_s|^2dA_s+\bar{\alpha}_1\int_t^T\|Z^{n}_s\|^2ds\right]\nonumber\\
&\leq & \E\left[\mu \int_{t}^{T}|Y^{n}_s|^2ds+(1+2\lambda)\int_t^T\rho(|Y^{n-1}|^2)ds\right]\nonumber\\
&&+C\E\left[|\xi|^2+\int_t^T(|f(s,0,0)|^2+\|h(s,0,0\|^2)ds+\int_t^T|g(s,0)|^2dA_s\right]. 
\end{eqnarray}
Therefore Gronwall's lemma yields 
\begin{eqnarray}\label{S3}
\E(|Y^{n}_t|^2)\leq \Gamma_t+ M\int_t^T\rho(\E(|Y^{n-1}_s|^2))ds,
\end{eqnarray}
where 
\begin{eqnarray*}
\Gamma_t=e^{\mu T}C\E\left[|\xi|^2+\int_t^T(|f(s,0,0)|^2+\|h(s,0,0\|^2)ds+\int_t^T|g(s,0)|^2dA_s\right]<+\infty
\end{eqnarray*}
and $M=(1+2\lambda)e^{\mu\,T}$. In view of \eqref{S3}, 
the proof of \eqref{S2} require an induction method.

For this, recalling again \eqref{S3}, since $Y^0 \equiv 0, \rho(0)=0$ and $({\bf H2}$-$(i)$, we have
\begin{eqnarray*}
\E(|Y^{1}_t|^2)&\leq &\Gamma_t+ M\int_t^T\rho(\E(|Y^{0}_s|^2)ds\\
&\leq & \Gamma_0\\
&<&+\infty.
\end{eqnarray*}
Since $\rho$ is non-decreasing, we have 
\begin{eqnarray*}
\E(|Y^{2}_t|^2)&\leq &\Gamma_t+ M\int_t^T\rho(\E(|Y^{1}_s|^2)ds\\
&\leq& \Gamma_0+ M(T-t)\rho(2\Gamma_0).
\end{eqnarray*}
Since, thanks to  $(\bf H2)$, $T\rho(2\Gamma_0)$ is finite, we can find $T_1\in[0,T)$ depends only on $\rho$ such that  $M(T-T_1)\rho(2\Gamma_0)\leq \Gamma_0$. Then for $t\in[T_1,T]$, we have
\begin{eqnarray*}
\E(|Y^{2}_t|^2)&\leq &\Gamma_0+ M(T-T_1)\rho(\Gamma_0)\\
&\leq & M_1,
\end{eqnarray*}
with $M_1=2\Gamma_0$.
We suppose that for some $n\in\N^*$, we have 
\begin{eqnarray}\label{Hr}
\E(|Y^{n}_t|^2)\leq M_1,\;\; \mbox{for all}\;\; t\in[T_1,T].
\end{eqnarray}
Next, it follows from inequalities \eqref{S3} and \eqref{Hr}, and the non decrease of $\rho$ that for $t\in[T_1,T]$
\begin{eqnarray*}
\E(|Y^{n+1}_t|^2)&\leq& \Gamma_0+ M\int_t^T\rho(\E(|Y^{n}_s|^2)ds\\
&\leq & \Gamma_0+ M(T-T_1)\rho(M_1)\\
&\leq& M_1,
\end{eqnarray*}
which according to \eqref{R1} prove Step 1.

{\it Step 2:} For $n,\,m\in\N$, we claim that there exist $\bar{M}>0$ such that
\begin{eqnarray}\label{S1}
\E\left[|Y^{n+m}_t-Y^n_t|^2+\int_t^T|Z^{n+m}_s-Z^n_s|^2ds\right]\leq \bar{M}\int_t^T\rho(\E(|Y^{n+m-1}_s-Y^{n-1}_s|^2)ds.
\end{eqnarray}
Indeed, for any $\bar{\mu}$, and in view of Itô's formula, we get
\begin{eqnarray*}
&&\E(e^{\bar{\mu} A_t}|Y^{n+m}_t-Y^n_t|^2)+\mu \int_t^Te^{\bar{\mu} A_s}|Y^{n+m}_s-Y^n_s|^2ds+\int_t^Te^{\bar{\mu} A_s}\|Z^{n+m}_s-Z^n_s\|^2ds\\
&=&2\E\left[\int_t^Te^{\bar{\mu} A_s}\langle Y^{n+m}_s-Y^n_s,\,f(s,Y^{n+m-1}_s,Z^{n+m}_s)-f(s,Y^{n-1}_s,Z^{n}_s)\rangle ds\right]\\
&&+2\E\left[\int_t^Te^{\bar{\mu} A_s}\langle Y^{n+m}_s-Y^n_s,\,g(s,Y^{n+m}_s)-g(s,Y^{n}_s)\rangle dA_s\right]\\
&&+2\E\left[\int_t^Te^{\bar{\mu} A_s}\|h(s,Y^{n+m-1}_s,Z^{n+m}_s)-h(s,Y^{n-1}_s,Z^{n}_s)\|^2ds\right]
\end{eqnarray*}
It follows from $(\bf H3)$ and use the same estimates as in step 1, that for all $\bar{\mu}>0$, 
\begin{eqnarray}\label{Z3}
&&\E(e^{\bar{\mu} A_t}|Y^{n+m}_t-Y^n_t|^2)+\bar{\alpha}\int_t^Te^{\bar{\mu} A_s}\|Z^{n+m}_s-Z^n_s\|^2ds\nonumber\\
&\leq &\E\left[\frac{1}{\lambda}\int_t^Te^{\mu A_s}|Y^{n+m}_s-Y^n_s|^2ds\right.\nonumber\\
&&+\left.(1+\lambda)\int_t^T\rho(e^{\mu A_s}|Y^{n+m-1}_s-Y^{n-1}_s|^2)ds\right],
\end{eqnarray}
where $\bar{\alpha}=1-\lambda K-\alpha$. Choosing $\lambda$ such that $\bar{\alpha}>0$, it follows from Gronwall's and Jensen's inequalities that
\begin{eqnarray*}
\E(e^{\bar{\mu} A_t}|Y^{n+m}_t-Y^n_t|^2)&\leq &\bar{M}\int_t^T\rho(\E(e^{\bar{\mu}A_s}|Y^{n+m-1}_s-Y^{n-1}_s|^2))ds,
\end{eqnarray*}
where $\bar{M}=e^{\frac{T}{\lambda}}(1+\lambda)$. Finally, passing to the limit when $\bar{\mu}$ tends to $0$ we obtain
\begin{eqnarray}\label{Step2}
\E(|Y^{n+m}_t-Y^n_t|^2)&\leq &\bar{M}\int_t^T\rho(\E(|Y^{n+m-1}_s-Y^{n-1}_s|^2))ds,
\end{eqnarray}
 \\
{\it Step 3:} The sequence of process $(Y^{n},Z^n)$ converge in $\mathcal{S}^{2}(\R^k)\times \mathcal{M}^2(\R^{d\times k})$\\ 
For $M$ and $M_1$ obtained in Step 1, let consider $(\phi_n)_{n\geq 1}$ the sequence of processes defined recursively by
\begin{eqnarray*}
\phi_1(t)=M\rho(M_1)(T-t),\;\;\; \;\;\;\phi_{n+1}(t)=M\int_{t}^{T}\rho(\phi_n(s))ds, \;\;\; t\in[0,T]
\end{eqnarray*}
and $(\phi_{n,m})_{n,m\geq 1}$ defined by
\begin{eqnarray*}
\phi_{n,m}(t)=\E(|Y^{n+m}_t-Y^n_t|^2),\;\;\; t\in[T_1,T].
\end{eqnarray*}
For any $m\geq 1$ and all $n\geq 1$, let us prove by induction method on $n$ that
\begin{eqnarray}\label{Z12}
\phi_{n,m}(t)\leq \phi_n(t)\leq \phi_{n-1}(t)\leq\cdot\cdot\cdot\leq \phi_1(t), \;\;\; t\in [T_1,T].
\end{eqnarray}
First, in view of Step 2 and Step 1 
\begin{eqnarray*}
\phi_{1,m}(t)=\E(|Y^{1+m}_t-Y^1_t|^2)&\leq &\bar{M}\int_t^T\rho(\E(|Y^m(s)|^2))ds\nonumber\\
&\leq & \bar{M} (T-t)\rho(M_1)\nonumber\\
&\leq & \bar{M}(T-t)\rho(M_1) 
\end{eqnarray*}
Assume that $\mu$ and $\lambda$  appear respectively  in Step 1 and 2 is defined such that $u=\frac{1}{\lambda}$, we have $\bar{M}\leq M$. Therefore it follows from the previous inequality that
\begin{eqnarray}\label{Z1}
\phi_{1,m}(t)&\leq & M(T-t)\rho(M_1)=\phi_1(t).
\end{eqnarray}
Now recall Step 2, it follows from \eqref{Step2} that
\begin{eqnarray*}
\phi_{2,m}(t)=\E(|Y^{2+m}_t-Y^2_t|^2)&\leq &\bar{M}\int_t^T\rho(\E(|Y^{1+m}(s)-Y^1(t)|^2))ds\\
&=&\bar{M}\int_t^T\rho(\phi_{1,m})ds\\
&\leq& M \int_t^T\rho(\phi_1(s))ds=\phi_2(t).
\end{eqnarray*}
On other hand, since $\phi_1(t)=M(T-t)\rho(M_1)\leq M_1$,
\begin{eqnarray*}
\phi_2(t)=M\int_t^T\rho(\phi_1(s))ds\leq M(T-t)\rho(M_1)=\phi_1(t).
\end{eqnarray*}
In the other word, we get
\begin{eqnarray*}
\phi_{2,m}(t)\leq \phi_{2}(t)\leq \phi_{1}(t), \;\; t\in [T_1,T].
\end{eqnarray*}
Next, let us assume that \eqref{Z12} holds for some $n\geq 2$. Using again Step 2, we get for all $t\in [T_1,T]$,

\begin{eqnarray*}
\phi_{n+1,m}(t)&\leq & M\int_t^T\rho(\phi_{n,m}(s))ds\\ &\leq & M\int_t^T\rho(\phi_{n}(s))ds=\phi_{n+1}(t)\\
&\leq & \int_t^T\rho(\phi_{n-1}(s))ds=\phi_n(t),
\end{eqnarray*}
which proved \eqref{Z12} for $n+1$. Therefore using induction principle, \eqref{Z12} holds. Moreover, for $t,t'\in [T_1,T]$, we have 
\begin{eqnarray}
\sup_{n\geq 0}|\phi_n(t)-\phi_n(t')|\leq M\rho(M_1)|t-t'|,
\end{eqnarray}
that proves that the decreasing sequence $(\phi_n)_{n\geq 0}$ is uniformly equicontinuous. Therefore, in view of Ascoli-Arzela theorem, it converge to a a continuous function $\phi$ as $n$ goes to $+\infty$ such that
\begin{eqnarray*}
\phi(t)=M\int_t^T\rho(\phi(s)ds.
\end{eqnarray*} 
Hence with the help of  Bihari's inequality we have $\phi\equiv 0$ on $[T_1,T]$. Finally according to \eqref{Z2} and \eqref{Z3} $(Y^n, Z^n)$ is a Cauchy sequence in $\mathcal{M}^2([T_1,T],\R^k)\times\mathcal{M}^2([T_1,T],\R^{d\times k})$. Furthermore by adapted calculus, 
\begin{eqnarray}
\E\left(\sup_{T_1\leq t\leq T}|Y^{n+m}(t)-Y^{n}(t)|^2\right)\leq C\E\left[\int_t^T|Y^{n+m}_s-Y^n_s|^2ds+\phi_n(T_1)\right.\nonumber\\
\left.+\int_t^T\|Z^{n+m}_s-Z^n_s\|^2ds\right],
\end{eqnarray} 
which implies that $(Y^n)$ is also a Cauchy sequence in a classical Banach space $\mathcal{S}^2([T_1,T],\R^k)$. Therefore there exist a process $(Y,Z)\in \mathcal{S}^2([T_1,T],\R^k)\times\mathcal{M}^2([T_1,T],\R^{d\times k})$ such that for $t\in[T_1,T]$, 
\begin{eqnarray}\label{S4}
\E\left[\sup_{0\leq t\leq T}|Y^{n}_t-Y_t|^{2}+\int_{t}^{T}\|Z^{n}_s-Z_s\|^2ds\right]\rightarrow 0\;\; \mbox{as}\;\; n\rightarrow +\infty.
\end{eqnarray}
{\it Step 4:} $(Y,Z)$ solves GBDSDE \eqref{EqGBSDE}\\
Letting $n\rightarrow +\infty$ in \eqref{eq5} provides that $(Y,Z)$ solves GBDSDE \eqref{EqGBSDE} on $[T_1,T]$. If $T_1=0$, then we have proved existence result. But if $T_1\neq 0$, we must continuous by study now the existence result for this equation
\begin{eqnarray}\label{S5}
Y_t=Y_{T_1}+\int_t^{T_1}f(s,Y_s,Z_s)ds+\int_t^{T_1}g(s,Y_s)dA_s+\int_t^{T_1}h(s,Y_s,Z_s)\overleftarrow{dB}_s-\int_t^{T_1}Z_sdW_s, \, 0\leq t\leq T_1.
\end{eqnarray} 
With the same analysis, there exist $T_2\in[0,T_1]$ and a process that also denoted $(Y,Z)$ defined on $[T_2,T_1]$ solution of \eqref{S5}. If $T_2=0$ the proof of existence is complete. Otherwise, we repeat the previous analysis. Therefore we built a decreasing sequence $(T_p)_{p\geq 1}$ such that on each $[T_{p},T_{p+1}]$, GBDSDE associated to the data $(Y_{p+1}, f,g,h)$ admit a solution. We also prove as like in \cite{ON} that there exists a finite $p$ such that $T_p=0$. Finally by the path continuity of  the solution on each interval, we obtain the solution on $[0,T]$ and end the proof of existence.\\
{\bf (ii) Uniqueness}\\
Let $(Y,Z)$ and $(\bar{Y},\bar{Z})$ be two solutions of GBDSDEs \eqref{EqGBSDE}. By use the now classical computation as above (Step 1 and Step 3) we have
\begin{eqnarray*}
\E\left[e^{\mu A_t}|Y_t-\bar{Y}_t|^2  +\int_{t}^Te^{\mu A_s}\|Z_s-\bar{Z}_s\|^{2}ds\right]\leq C\int_t^T\rho(\E(e^{\mu A_s}|Y_s-\bar{Y}_s|^2))ds,
\end{eqnarray*}
which yields by Bihari's inequality $Y=\bar{Y}$ and also $Z=\bar{Z}$. 
\end{proof}

\begin{remark}
Let consider the following assumption
\begin{description}
\item [(H3')]
\begin{equation*}
\left\{
\begin{array}{lll}
(i)&  f(.,0,0) \in\mathcal{M}^{2}(\R), \ \ h(.,0,0) \in
\mathcal{M}^{2}(\R^{l}),\; g(s,0)\in \mathcal{A}^{2}(\R),\\\\ 
(ii)&\mid f(t,y_{1},z_{1})-f(t,y_{2},z_{2}) \mid^{2} \leq \rho(t,\mid
y_{1}-y_{2}\mid^{2})+C\mid \mid z_{1}-z_{2}\mid \mid^{2} \vspace{0.2cm} &
\\\\
(iii)&\mid \mid h(t,y_{1},z_{1})-h(t,y_{2},z_{2}) \mid \mid^{2} \leq \rho(t,\mid
y_{1}-y_{2}\mid^{2})+\alpha \mid \mid z_{1}-z_{2}\mid \mid^{2} &\\\\
(iv)&\langle y_1-y_2,(g(t,y_1)-g(t,y_2)\rangle \leq \beta |y_1-y_2|^{2}&
\end{array}
\right.,
\end{equation*}
for all $(t, y_{i},z_{i})\in [0,T]\times\mathbb{R}%
\times \mathbb{R}^{d}$ $i=1,2$, where $C > 0,\; 0 < \alpha < 1$ and $\beta\in\R$ are three constants and $\rho:[0,T]\times \R^{+}\rightarrow\R^{+}$ satisfies:
\begin{itemize} 
\item[(a)] for fixed $t\in[0,T],\, \rho (t,.)$ is continuous, concave and non-decreasing such that $\rho(0)=0$
\item[(b)] for all $\displaystyle u\geq 0,\; \int_0^T\rho(t,u)dt<+\infty$,
\item[(c)] For any $M>0$, the following ODE
 \begin{eqnarray*}
 \left\{
 \begin{array}{ll}
 u'(t)=-M\rho(t,u)&\\\\
 u(T)=0&
 \end{array}
 \right.
\end{eqnarray*}
\end{itemize}
has a unique solution $u(t)= 0,\, t\in[0,T]$.
\end{description} 
Theorem \ref{Tfirst} remain valid replacing $(\bf H3)$ by $(\bf H3')$. Obviously, this change requires adaptation in proof. Indeed, since the function $\rho$ depends also to $t$, we cannot apply which need to be  replaced by $(\bf H3')$-$(c)$. We refer the reader to the works of Owo and N'zi \cite{ON} for more details.
\end{remark}

We give now the main result of this section.

\begin{theorem}\label{thm3.1} 
Assume the assumptions $(\bf H1)$-$(\bf H4)$ hold. Then, VGBDSDEs
\eqref{equation1} has a unique solution.
\end{theorem}

\begin{proof}
{(\bf i)} {\bf Existence}\\
Let $\nabla \varphi_\varepsilon$ (resp. $\nabla \psi_\varepsilon$) be Yosida approximation of the sub-differential operator $\partial \varphi$ (resp. $\partial \varphi$ and following GBDSDE 
\begin{eqnarray}\label{eq1}
Y_t^\varepsilon
+\int_t^T\nabla \varphi_\varepsilon(Y_s^\varepsilon)ds
+\int_t^T\nabla \psi_\varepsilon(Y_s^\varepsilon)dA_s
&=&\xi+\int_t^Tf(s,Y_s^\varepsilon,Z_s^\varepsilon)ds
+\int_t^Tg(s,Y_s^\varepsilon)\,{\rm d}A_s\nonumber\\
&&+\int_t^Th(s,Y_s^\varepsilon,Z_s^\varepsilon)\overleftarrow{dB}_s-\int_t^TZ_s^\varepsilon dW_s,
\end{eqnarray}
where $\xi,\, f,\, h ,\, g,\, \varphi$ and $\psi$ satisfy $(\bf H1)$-$(\bf H4)$. 
Therefore in virtue of Theorem \ref{Tfirst}, GBDSDE \eqref{eq1} admits a unique solution $(Y^{\varepsilon},Z^{\varepsilon})$. Moreover, since $\nabla\psi_{\varepsilon}$ and $\nabla\varphi_{\varepsilon}$ are monotone, it follows from comparison that $(Y^{\varepsilon},Z^{\varepsilon})_{\varepsilon}$is monotone family processes 

Our goal is to provide that the family of processes $(Y^{\varepsilon},Z^{\varepsilon})_{\varepsilon>0}$ converges and its limit is   solution of  BDSGVI \eqref{equation1}. In the sequel, $C>0$ is a constant which can change its value from
line to line.\\
{\it Step 1: A first priori estimate} 
\begin{eqnarray}\label{AStep1}
\E\left[\sup_{0\leq t\leq T}|Y_t^\varepsilon|^2+\int_0^T(|Y^{\varepsilon}_t|^2dA_t+\|Z_t^\varepsilon\|^2dt)\right]\leq C.
\end{eqnarray}
By using a generalized Itô's formula (cf. \cite{PP} Lemma 1.3) and taking expectation, we get,
\begin{eqnarray*}
&&\E\left[|Y_t^\varepsilon|^2 +\int_t^T\|Z_s^\varepsilon\|^2ds +\int_t^T\langle Y_s^\varepsilon,\nabla
\varphi_\varepsilon(Y_s^\varepsilon)\rangle ds+\int_t^T\langle Y_s^\varepsilon,\nabla
\psi_\varepsilon(Y_s^\varepsilon) \rangle dA_s\right]\\
&=&\E\left[|\xi|^2+2\int_t^T\langle Y_s^\varepsilon,\,f(s,Y_s^\varepsilon,Z_s^\varepsilon)\rangle ds+2\int_t^T\langle Y_s^\varepsilon,\, g(s,Y_s^\varepsilon)\rangle dA_s +2\int_t^T\|h(s,Y_s^\varepsilon,Z_s^\varepsilon)\|^2ds\right].
\end{eqnarray*}
From $(iv)$ of Proposition \ref{pro2.1}, we obtain
\begin{eqnarray*}
\E\left[|Y_t^\varepsilon|^2 +\int_t^T\|Z_s^\varepsilon\|^2ds\right]
&\leq&\E\left[|\xi|^2+2\int_t^T\langle Y_s^\varepsilon,\, f(s,Y_s^\varepsilon,Z_s^\varepsilon) \rangle ds+2\int_t^T\langle Y_s^\varepsilon,\, g(s,Y_s^\varepsilon) \rangle dA_s \right.\\
&&\left.+2\int_t^T\|h(s,Y_s^\varepsilon,Z_s^\varepsilon)\|^2ds\right]
\end{eqnarray*}
On the other hand, Schwartz's inequality and assumptions $(\bf H3)$ imply that for all $r, r'>0$
\begin{eqnarray*}
2\langle y,\, f(s,y,z)\rangle &=&\langle y,\, f(s,y,z)-f(s,0,0)\rangle+\langle y,\, f(s,0,0)\rangle\\
&\leq &r\rho(|y|^2)+(\frac{K}{r}+1)|y|^2+rK|z|^2+|f(s,0,0)|^{2},
\end{eqnarray*}
\begin{eqnarray*}
2\langle y,g(s,y)\rangle &=&2\langle y, \, g(s,y)-g(s,0)\rangle+2\langle y,\,g(s,0)\rangle\\
&\leq & (2\beta+|\beta|)|y|^{2}+\frac{1}{|\beta|}|g(s,0)|^{2}
\end{eqnarray*}
and 
\begin{eqnarray*}
\|h(s,y,z)\|^{2}&\leq & (1+\frac{1}{r'})\|h(s,y,z)-h(s,0,0)\|^{2}+(1+r')\|h(s,0,0)\|^{2}\\
&\leq&(1+\frac{1}{r'})\rho(|y|^{2})+\alpha(1+\frac{1}{r'})\|z\|^{2}+(1+r')\|h(s,0,0)\|^{2}.
\end{eqnarray*}
Chosing $r=\frac{1-\alpha}{2K}$ and $r'=\frac{3\alpha}{1-\alpha}$ and since $\beta<0$, we get
\begin{eqnarray*}
&&\E\left[|Y_t^\varepsilon|^2+|\beta|\int_{t}^{T}|Y^{\varepsilon}_s|^2 dA_s+\frac{1-\alpha}{6}\int_t^T\|Z_s^\varepsilon\|^2 ds\right]\nonumber\\
&\leq&
C\E\left[|\xi|^{2}+\int_t^T(|Y_s^\varepsilon|^2+\rho(Y_s^\varepsilon|^2))ds\right.\\
&&+\left.\int_t^T(|f(s,0,0)|^{2}+\|h(s,0,0)\|^{2})ds+\int_t^T|g(s,0)|^{2}dA_s\right]. 
\end{eqnarray*}
Finally, Bihari-LaSalle inequality yields
\begin{eqnarray*}
&&\sup_{0\leq t\leq T}\E(|Y_t^\varepsilon|^2)+C\E\left(\int_t^T|Y_t^\varepsilon|^2 dA_t+\int_t^T\|Z_t^\varepsilon\|^2 dt\right)\\
&\leq& C\E\left[|\xi|^{2}+\int_t^T(|f(s,0,0)|^{2}+\|h(s,0,0)\|^{2})ds+\int_t^T|g(s,0)|^{2}dA_s\right]. 
\end{eqnarray*}
Furthermore, applying again generalized Itô's formula together with Burkholder-Davis-
Gundy's inequality, we obtain
\begin{eqnarray}\label{equation3}
\E\left(\sup_{0\leq t\leq T}|Y_t^\varepsilon|^2\right)&\leq &C\E\left[|\xi|^{2}+\int_0^T(|f(s,0,0)|^{2}+\|h(s,0,0)\|^{2})ds+\int_0^T|g(s,0)|^{2}dA_s\right]\nonumber\\
&&+C\E\left(\int_0^T|Y^{\varepsilon}|^{2}\|Z_t^\varepsilon\|^2 dt\right)^{1/2}
+C\E\left(\int_0^T|Y^{\varepsilon}|^{2}\|h(s,Y_s^\varepsilon,Z_s^\varepsilon)\|^2 dt\right)^{1/2}.
\end{eqnarray}
We estimate the last term as follows
\begin{eqnarray*}
&&C\E\left(\int_0^T|Y^{\varepsilon}_s|^{2}\|h(s,Y_s^\varepsilon,Z_s^\varepsilon)\|^2 ds\right)^{1/2}\\
&\leq &\frac{1}{4}\E\left(\sup_{0\leq t\leq T}|Y_s^\varepsilon|^2\right)+C\E\left(\int_0^T\|h(s,Y_s^\varepsilon,Z_s^\varepsilon)\|^2 ds\right)\\
&\leq & \frac{1}{4}\E\left(\sup_{0\leq t\leq T}|Y_s^\varepsilon|^2\right)+C\E\left[|\xi|^{2}+\int_0^T(|f(s,0,0)|^{2}+\|h(s,0,0)\|^{2})ds+\int_0^T|g(s,0)|^{2}dA_s\right]\\
&&+C\int_{0}^{T}\rho(\E(\sup_{0\leq u\leq s}|Y_u^\varepsilon|^2))ds.
\end{eqnarray*}
The second term on the right side of \eqref{equation3} is treated analogously to obtain
\begin{eqnarray*}
&&C\E\left(\int_0^T|Y^{\varepsilon}|^{2}\|Z_t^\varepsilon\|^2 dt\right)^{1/2}\\
&\leq&  \frac{1}{4}\E\left(\sup_{0\leq t\leq T}|Y_s^\varepsilon|^2\right)\\
&&+C\E\left[|\xi|^{2}+\int_0^T(|f(s,0,0)|^{2}+\|h(s,0,0)\|^{2})ds+\int_0^T|g(s,0)|^{2}dA_s\right].
\end{eqnarray*}
Therefore we deduce 
\begin{eqnarray*}
\E\left(\sup_{0\leq t\leq T}|Y_t^\varepsilon|^2\right)&\leq& C\E\left[|\xi|^{2}+\int_0^T(|f(s,0,0)|^{2}+\|h(s,0,0)\|^{2})ds+\int_0^T|g(s,0)|^{2}dA_s\right]\\
&&+\int_{0}^{T}\rho(\E(\sup_{0\leq u\leq s}|Y_u^\varepsilon|^2))ds,
\end{eqnarray*}
which by using again Bihari inequality provides
\begin{eqnarray*}
\E\left(\sup_{0\leq t\leq T}|Y_t^\varepsilon|^2\right)\leq C\E\left[|\xi|^{2}+\int_0^T(|f(s,0,0)|^{2}+\|h(s,0,0)\|^{2})ds+\int_0^T|g(s,0)|^{2}ds\right].
\end{eqnarray*}
Finally according to $(\bf H1)$ and $(\bf H2)$ we get \eqref{AStep1}.\\
{\it Step 2: A second priori estimate}\\
We have for all $t\in [0,T]$,  
\begin{eqnarray}\label{estisec}
\begin{array}{lll}
(a)&\displaystyle \E\left(\int_0^T|\nabla \varphi_\varepsilon(Y_t^\varepsilon)|^2dt+\int_0^T|\nabla\psi_\varepsilon(Y_t^\varepsilon)|^2dA_t\right)\leq C\Lambda,\\\\
(b)&\displaystyle \E\left(\varphi(J_\varepsilon(Y_t^\varepsilon))+
\psi(\bar{J}_\varepsilon(Y_t^\varepsilon))+\int_0^T\varphi(J_\varepsilon(Y_t^\varepsilon))dt+
\int_0^T\psi(\bar{J}_\varepsilon(Y_t^\varepsilon))dA_t\right)\leq C\Lambda,\\\\
(c)&\displaystyle \E\left(|Y_t^\varepsilon-J_\varepsilon(Y_t^\varepsilon)|^2+
|Y_t^\varepsilon-\bar{J}_\varepsilon(Y_t^\varepsilon)|^2\right)\leq
\varepsilon C\Lambda, \\\\
\end{array}
\end{eqnarray}
where 
\begin{eqnarray*}
	\Lambda=\E\left(|\xi|^2+\varphi(\xi)+\psi(\xi)+\int_0^T|\gamma_t|^2dt+\int_0^T|\eta_t|^2dA_t\right).
\end{eqnarray*}
Like in Pardoux and  R\u{a}\c{s}canu in \cite{PR} (see Proposition 2.2), For $t=t_0<t_1<\cdot\cdot\cdot\cdot<t_n=T$, where $t_i-t_{i-1}=\frac{1}{n}$, let write the subdifferential inequality
\begin{eqnarray*}
\varphi_{\varepsilon}(Y^{\varepsilon}_{t_{i+1}})\geq  \varphi_{\varepsilon}(Y^{\varepsilon}_{t_{i}})+\langle\nabla\varphi_{\varepsilon}(Y^{\varepsilon}_{t_{i}}),\,Y^{\varepsilon}_{t_{i+1}}-Y^{\varepsilon}_{t_{i}}\rangle
\end{eqnarray*}
and
\begin{eqnarray*}
\psi_{\varepsilon}(Y^{\varepsilon}_{t_{i+1}})\geq \psi_{\varepsilon}(Y^{\varepsilon}_{t_{i}})+\langle\nabla\psi_{\varepsilon}(Y^{\varepsilon}_{t_{i}}),\,Y^{\varepsilon}_{t_{i+1}}-Y^{\varepsilon}_{t_{i}}\rangle.
\end{eqnarray*}
Using the approximation equation \eqref{eq1}, summing up over $i$, and passing to the limit as $n$ goes to $+\infty$, we deduce
\begin{eqnarray}\label{seconesti}
&&\varphi_\varepsilon(Y_t^\varepsilon)
+\psi_\varepsilon(Y_t^\varepsilon)+\int_t^T|\nabla \varphi_\varepsilon(Y_s^\varepsilon)|^2 ds+
\int_t^T|\nabla \psi_\varepsilon(Y_s^\varepsilon)|^2dA_s\nonumber\\
&&+\int_t^T\langle\nabla \ \varphi_\varepsilon(Y_s^\varepsilon),\,\nabla
\psi_\varepsilon(Y_s^\varepsilon)\rangle(ds+dA_s)\nonumber\\
&\leq&\varphi_\varepsilon(\xi)+\psi_\varepsilon(\xi)+
\int_t^T\langle \nabla \varphi_\varepsilon(Y_s^\varepsilon)+\nabla
\psi_\varepsilon(Y_s^\varepsilon),\,f(s,Y_s^\varepsilon,Z_s^\varepsilon) \rangle ds
\nonumber\\
&&+ \int_t^T\langle \nabla\varphi_\varepsilon(Y_s^\varepsilon)+\nabla
\psi_\varepsilon(Y_s^\varepsilon)),\,g(s,Y_s^\varepsilon)\rangle dA_s+ \int_t^T\langle \nabla
\varphi_\varepsilon(Y_s^\varepsilon)+\nabla
\psi_\varepsilon(Y_s^\varepsilon)),\,h(s,Y_s^\varepsilon,Z_s^\varepsilon)\overleftarrow{dB}_s\rangle\nonumber\\
&&-\int_t^T\langle \nabla\varphi_\varepsilon(Y_s^\varepsilon)+\nabla\psi_\varepsilon(Y_s^\varepsilon),\,Z_s^\varepsilon dW_s\rangle.
\end{eqnarray}
From $(\bf H3)$, for $\theta =\varphi\;\mbox{or}\; \psi$ we get
\begin{eqnarray*}
\langle \nabla\theta_\varepsilon(y),\,f(s,y,z)\rangle \leq \frac{1}{4}|\nabla
\theta_\varepsilon(y)|^2+K(\gamma^2_s+|y|^2+\|z\|^2)
\end{eqnarray*}
and
\begin{eqnarray*}
\langle\nabla \theta_\varepsilon(y),\,g(s,y)\rangle\leq \frac{1}{4}|\nabla
\theta_\varepsilon(y)|^2+K(\eta^2_s+|y|^2).
\end{eqnarray*}
Hence taking expectation in \eqref{seconesti}, we obtain
\begin{eqnarray}\label{estia}
&&\E\left[\int_t^T|\nabla \varphi_\varepsilon(Y_s^\varepsilon)|^2ds+|\nabla \psi_\varepsilon(Y_s^\varepsilon)|^2dA_s)\right]\nonumber\\
&\leq&
C\E\left[\varphi(\xi)+\psi(\xi)+\sup_{0\leq t\leq T}|Y^{\varepsilon}_t|^2+\int_0^T\|Z^{\varepsilon}_t\|^2dt\right.\nonumber\\
&&\left.+\int_0^T|Y^{\varepsilon}_t|^2dA_t+\int_0^T\gamma^2_sds+\int_0^T\eta_s^2dA_s\right],
\end{eqnarray}
where we have used the fact that $\varphi_{\varepsilon}(Y^{\varepsilon}_t)+\psi_{\varepsilon}(Y^{\varepsilon}_t)>0$ and 
$\varphi_{\varepsilon}(\xi)+\psi_{\varepsilon}(\xi)\leq \varepsilon(\varphi(\xi)+\psi(\xi))$ successively. Therefore $(a)$ holds by combining \eqref{estia} with the result of Step 1.\\
$(b)$ From \eqref{seconesti}, we have 
\begin{eqnarray*}
\E\left[\varphi_\varepsilon(Y_s^\varepsilon)+\psi_\varepsilon(Y_s^\varepsilon)\right]&\leq& \E\left[\varepsilon(\varphi(\xi)+\psi(\xi))+\left(\frac{2}{\delta_1}+\frac{2}{\delta_2}\right)\int_t^T(|\nabla \varphi_\varepsilon(Y_s^\varepsilon)|^2ds+|\nabla \psi_\varepsilon(Y_s^\varepsilon)dA_s|^2)\right.\\
&&\left.+\delta_1\int_t^T|f(s,Y^{\varepsilon}_s,Z^{\varepsilon}_s)|^2ds+\delta_2\int_t^T|g(s,Y^{\varepsilon}_s)|^2dA_s\right].
\end{eqnarray*}
Since $\varphi(J^{\varepsilon}(Y^{\varepsilon}_t))\leq \varphi_{\varepsilon}(Y^{\varepsilon}_t)$ and 
$\psi(\bar{J}^{\varepsilon}(Y^{\varepsilon}_t))\leq \psi_{\varepsilon}(Y^{\varepsilon}_t)$, by the same arguments as before we obtain
\begin{eqnarray*}
\E[\varphi(J^{\varepsilon}(Y^{\varepsilon}_t))+\psi(\bar{J}^{\varepsilon}(Y^{\varepsilon}_t))]&\leq &C\E\left[\varphi(\xi)+\psi(\xi)+|\xi|^{2}+\int_0^T|f(s,0,0)|^2 ds+\int_0^T|g(s,0)|^2dA_s\right.\\
&&\left.+\int_0^T\|h(s,0,0)\|^2 ds+\int_0^T\gamma^2_sds+\int_0^T\eta^2 _sdA_s\right].
\end{eqnarray*}
The assumption $(c)$ is proved similarly since $\frac{1}{2\varepsilon}|J^{\varepsilon}(y)-y|^{2}\leq \varphi_{\varepsilon}(y)$ and $\frac{1}{2\varepsilon}|\bar{J}^{\varepsilon}(y)-y|^{2}\leq \psi_{\varepsilon}(y)$.\\
{\it Step 3: Convergence of the family of processes $(Y^{\varepsilon},Z^{\varepsilon})$}.\\
For arbitrary $\varepsilon,\, \delta >0$, we have
\begin{eqnarray}\label{equation2}
\E\left[\sup_{0\leq t\leq T}|Y_t^\varepsilon-Y_t^\delta|^2+\int^{T}_{0}\|Z_t^\varepsilon-Z_t^\delta\|^2dt\right]\leq C(\varepsilon+\delta)\Lambda.
\end{eqnarray}
Indeed, by generalized Itô's formula,
\begin{eqnarray*}
&&|Y_t^\varepsilon-Y_t^\delta|^2+2\int_{t}^{T}\langle Y^{\varepsilon}_s-Y^{\delta}_s,\,\nabla\varphi_{\varepsilon}(Y^{\varepsilon}_s)-\nabla\varphi_{\delta}(Y^{\delta}_s)\rangle ds\\&&+2\int_{t}^{T}\langle Y^{\varepsilon}_s-Y^{\delta}_s,\, \nabla\psi_{\varepsilon}(Y^{\varepsilon}_s)-\nabla\psi_{\delta}(Y^{\delta}_s)\rangle dA_s\\
&=& 2\int_{t}^{T}\langle Y^{\varepsilon}_s-Y^{\delta}_s,\, f(s,Y^{\varepsilon}_s,Z^{\varepsilon}_s)-f(s,Y^{\delta}_s,Z^{\delta}_s)\rangle ds+2\int_{t}^{T}\langle Y^{\varepsilon}_s-Y^{\delta}_s,\, g(s,Y^{\varepsilon}_s)-g(s,Y^{\delta}_s)\rangle  dA_s\\
&&-\int_{t}^{T}\|Z^{\varepsilon}_s-Z^{\delta}_s\|^2ds+\int_{t}^{T}\|h(s,Y^{\varepsilon}_s,Z^{\varepsilon}_s)-h(s,Y^{\delta}_s,Z^{\delta}_s\|^2ds-2\int_{t}^{T} \langle Y^{\varepsilon}_s-Y^{\delta}_s,\,(Z^{\varepsilon}_s-Z^{\delta}_s)dW_s\rangle\\
&&+2\int_{t}^{T}\langle Y^{\varepsilon}_s-Y^{\delta}_s,\, (h(s,Y^{\varepsilon}_s,Z^{\varepsilon}_s)-h(s,Y^{\delta}_s,Z^{\delta}_s))\overleftarrow{dB}_s\rangle.
\end{eqnarray*}
On the other hand by assumption $(\bf H3)$, we have
\begin{eqnarray*}
\begin{array}{lll}
2\langle Y^{\varepsilon}_s-Y^{\delta}_s),\,f(s,Y^{\varepsilon}_s,Z^{\varepsilon}_s)-f(s,Y^{\delta}_s,Z^{\delta}_s)\rangle \leq 
\frac{1}{r}|Y^{\varepsilon}_s-Y^{\delta}_s|^2+r\rho(|Y^{\varepsilon}_s-Y^{\delta}_s|^2)+rC\|Z^{\varepsilon}_s-Z^{\delta}_s\|^2,&\\\\\
2\langle Y^{\varepsilon}_s-Y^{\delta}_s,\,g(s,Y^{\varepsilon}_s)-g(s,Y^{\delta}_s)\rangle\leq 
\beta|Y^{\varepsilon}_s-Y^{\delta}_s|^2&\\
\mbox{and} &\\
\|h(s,Y^{\varepsilon}_s,Z^{\varepsilon}_s)-h(s,Y^{\delta}_s,Z^{\delta}_s\|\leq \rho(|Y^{\varepsilon}_s-Y^{\delta}_s|^2)+\alpha\|Z^{\varepsilon}_s-Z^{\delta}_s\|^2.&
\end{array}
\end{eqnarray*}
Since $\beta<0$, if we take $r=\frac{1-\alpha}{2C}$, we obtain
\begin{eqnarray*}
&&|Y_t^\varepsilon-Y_t^\delta|^2+\int_{t}^{T}\|Z^{\varepsilon}_s-Z^{\delta}_s\|^2ds\\
&\leq& \int_{t}^{T}\rho(|Y^{\varepsilon}_s-Y^{\delta}_s|^2)ds-2\int_{t}^{T}\langle Y^{\varepsilon}_s-Y^{\delta}_s,\,\nabla\varphi_{\varepsilon}(Y^{\varepsilon}_s)-\nabla\varphi_{\delta}(Y^{\delta}_s)\rangle ds\\
&&-2\int_{t}^{T}\langle Y^{\varepsilon}_s-Y^{\delta}_s,\,\nabla\psi_{\varepsilon}(Y^{\varepsilon}_s)-\nabla\psi_{\delta}(Y^{\delta}_s)\rangle dA_s
-2\int_{t}^{T}\langle Y^{\varepsilon}_s-Y^{\delta}_s,\,(Z^{\varepsilon}_s-Z^{\delta}_s)dW_s\rangle\\
&&+2\int_{t}^{T} \langle Y^{\varepsilon}_s-Y^{\delta}_s),\,(h(s,Y^{\varepsilon}_s,Z^{\varepsilon}_s)-h(s,Y^{\delta}_s,Z^{\delta}_s))\overleftarrow{dB}_s\rangle.
\end{eqnarray*}
By $(v)$ of Proposition \ref{pro2.1} it follows that
\begin{eqnarray*}
&&|Y_t^\varepsilon-Y_t^\delta|^2+\int_{t}^{T}\|Z^{\varepsilon}_s-Z^{\delta}_s\|^2ds\\
&\leq& \int_{t}^{T}\rho(|Y^{\varepsilon}_s-Y^{\delta}_s|^2)ds+2(\varepsilon+\delta)\int_{t}^{T}|\nabla \varphi_{\varepsilon}(Y^{\varepsilon}_s)||\nabla\varphi_{\delta}(Y^{\delta}_s)|ds\\
&&+2(\varepsilon+\delta)\int_{t}^{T}|\nabla \psi_{\varepsilon}(Y^{\varepsilon}_s)||\nabla\psi_{\delta}(Y^{\delta}_s)|dA_s-2\int_{t}^{T} \langle Y^{\varepsilon}_s-Y^{\delta}_s,\, (Z^{\varepsilon}_s-Z^{\delta}_s)dW_s\rangle\\
&&+2\int_{t}^{T}\langle Y^{\varepsilon}_s-Y^{\delta}_s,\,(h(s,Y^{\varepsilon}_s,Z^{\varepsilon}_s)-h(s,Y^{\delta}_s,Z^{\delta}_s))\overleftarrow{dB}_s\rangle.
\end{eqnarray*}
Now, from Step 2 $(a)$, we have
\begin{eqnarray}
2(\varepsilon+\delta)\left[\int_{t}^{T}|\nabla \varphi_{\varepsilon}(Y^{\varepsilon}_s)||\nabla\varphi_{\delta}(Y^{\delta}_s)|ds+\int_t^T|\nabla \psi_{\varepsilon}(Y^{\varepsilon}_s)||\nabla\psi_{\delta}(Y^{\delta}_s)|dA_s)\right]\leq (\varepsilon+\delta)\Lambda,
\end{eqnarray}
which implies that
\begin{eqnarray*}
\E(|Y_t^\varepsilon-Y_t^\delta|^2)+\E\left(\int_{0}^{T}\|Z^{\varepsilon}_s-Z^{\delta}_s\|^2ds\right)
&\leq & \int_{t}^{T}\rho(\E(|Y^{\varepsilon}_s-Y^{\delta}_s|^2))ds+(\varepsilon+\delta)\Lambda.
\end{eqnarray*}
Bihari-Lasalle inequality yields
\begin{eqnarray}
\sup_{0\leq t\leq T}\E(|Y_t^\varepsilon-Y_t^\delta|^2)+\E\left(\int_{0}^{T}\|Z^{\varepsilon}_s-Z^{\delta}_s\|^2ds\right)\leq C(\varepsilon+\delta)\Lambda.
\end{eqnarray}
The result follows from Burkholder-Davis-Gundy's inequality. Therefore the family of process $(Y^{\varepsilon},Z^{\varepsilon})$ is Cauchy family   in a Banach space $\mathcal{S}^{2}(\R^k)\times\mathcal{M}^{2}(\R^{d\times k})$. Then there exist $(Y,Z)\in \mathcal{S}^{2}(\R^k)\times\mathcal{M}^{2}(\R^{d\times k})$ such that $\displaystyle \lim_{\varepsilon\rightarrow 0}(Y^{\varepsilon}_t,Z^{\varepsilon}_t)=(Y_t,Z_t)$. Furthermore, from Step 2 we have
\begin{eqnarray}\label{conv}
\lim_{\varepsilon\rightarrow 0}J^{\varepsilon}(Y^{\varepsilon})=Y,\; \mbox{in}\; \mathcal{M}^{2}(\R^{k})\; \mbox{and}\;
\lim_{\varepsilon\rightarrow 0}\E(|J^{\varepsilon}(Y^{\varepsilon}_t)-Y_t|^2)=0, \,\; t\in[0,T].
\end{eqnarray}
Furthermore, for all let us set
\begin{eqnarray*}
\bar{U}^{\varepsilon}_t=\int_0^tU^{\varepsilon}_sds\;\;\;\; \mbox{and}\;\;\;\; 
\bar{V}^{\varepsilon}_t =\int_0^t V^{\varepsilon}_sdA_s,
\end{eqnarray*}
where $U^{\varepsilon}_t=\nabla\varphi_{\varepsilon}(Y^{\varepsilon}_t)$ and $V^{\varepsilon}_t=\nabla\varphi_{\varepsilon}(Y^{\varepsilon}_t)$.

It follows from approximating equation \eqref{eq1} and Step 3 that
\begin{eqnarray*}
\E\left[\sup_{0\leq t\leq T}(|\bar{U}^{\varepsilon}_t-\bar{U}^{\delta}_t|^2+|\bar{V}^{\varepsilon}_t-\bar{V}^{\delta}_t|^2)\right]\leq \E\left[\sup_{0\leq t\leq T}|Y_t^\varepsilon-Y_t^\delta|^2+\int^{T}_{0}\|Z_t^\varepsilon-Z_t^\delta\|^2dt\right].
\end{eqnarray*} 
Since the right side of above converges to $0$ as $\varepsilon,\, \delta\rightarrow 0$, $(\bar{U}^{\varepsilon},\bar{V}^{\varepsilon})$ is a Cauchy family and there exists a measurable process $(\bar{U}_t,\bar{V}_t)_{0\leq t\leq T}$ such that
\begin{eqnarray*}
\E\left[\sup_{0\leq t\leq T}(|\bar{U}^{\varepsilon}_t-\bar{U}_t|^2+\int_0^T|\bar{V}^{\varepsilon}_t-\bar{V}_t|^2)\right]\rightarrow 0\; \mbox{as}\; \varepsilon\rightarrow 0.
\end{eqnarray*}
{\it Step 4: A quadruplet of process $(Y,Z,U,V)$ solves the MGBDSDE \eqref{equation1}}.\\
From Step 2 there exists a constant independent of $\varepsilon$ such that
\begin{eqnarray*}
&&\sup_{\varepsilon}\E\left[\int_0^T|U^{\varepsilon}_t|^2dt+\int_0^T|V^{\varepsilon}_t|^2dA_t\right]\\
&\leq& \sup_{\varepsilon}\E\left[\int_0^T|\nabla\varphi_{\varepsilon}(Y^{\varepsilon}_t)|^2dt+|\nabla\psi_{\varepsilon}(Y^{\varepsilon}_t)|^2dA_t\right]\\
&\leq & C,
\end{eqnarray*}
which means that $\bar{U}^{\varepsilon}$ (resp. $\bar{V}^{\varepsilon}$) is  bounded in the Sobolev space $L^2(\Omega, H^{1}([0,T],dt))$  (resp. $L^2(\Omega, H^{1}([0,T],dA_t))$). Therefore converge weakly to $\bar{U}$ and $\bar{V}$ respectively.
In particular, $\bar{U}$ (resp. $\bar{V}$) is absolutely continuous which respect $dt$ (resp. $dA_t$), i.e. there exist two progressively measurable processes $U$ and $V$ such that
\begin{eqnarray*}
\bar{U}_t=\int_0^tU_sds\;\;\; \mbox{and}\;\; \; \bar{V}_t=\int_0^tV_sdA_s.
\end{eqnarray*}
Moreover by Fatou's Lemma, $U$ and $V$ belong respectively in $\mathcal{M}^{2}(\R^k)$ and $\mathcal{A}^2(\R^k)$. Assertion $(i)$ in Definition \ref{def2.1} is then proved. Let now prove $(ii)$ i.e. $(Y_t,U_t)\in\partial\varphi,\, d\P\otimes dt$-a.e and    
 $(Y_t,V_t)\in\partial\psi,\, d\P\otimes dA_t$-a.e on $[0,T]$. For this instance, for $a<b\leq T,\, u\in\mathcal{M}^{2}(\R^{k})$ and $v\in\mathcal{A}^{2}(\R^k)$, since $\bar{U}^{\varepsilon},\, \bar{V}^{\varepsilon}$ and $Y^{\varepsilon}$ converge uniformly to $\bar{U},\, \bar{V}$ and $Y$, we deduce from Lemma 5.8 in Gegout-Petit and Pardoux \cite{Geg} that
\begin{eqnarray}\label{ii1}
\begin{array}{rrr}
\displaystyle \int_a^b \langle U^{\varepsilon}_t,\, u_t-Y_t^{\varepsilon}\rangle dt\rightarrow\int_a^b\langle U_t,\, u_t-Y_t\rangle dt&\\
\mbox{and}&\\
\displaystyle \int_a^b \langle V^{\varepsilon}_t,\, v_t-Y_t^{\varepsilon}\rangle dA_t\rightarrow\int_a^b \langle V_t, v_t-Y_t\rangle  dA_t& 
\end{array}
\end{eqnarray}
as $\varepsilon$ goes to $0$
in probability, which together with \eqref{conv} imply
\begin{eqnarray}\label{ii2}
\begin{array}{rrr}
\displaystyle \int_a^b \langle U^{\varepsilon}_t,\,J^{\varepsilon}(Y_t^{\varepsilon})-Y_t^{\varepsilon}\rangle dt\rightarrow 0 &\\
\mbox{and}&\\
\displaystyle \int_a^b \langle V^{\varepsilon}_t,\, \bar{J}^{\varepsilon}(Y_t^{\varepsilon})-Y_t^{\varepsilon}\rangle dA_t\rightarrow 0& 
\end{array}
\end{eqnarray}
as $\varepsilon$ goes to $0$.
On the other hand, in virtue of $(ii)$ of Proposition \ref{pro2.1}, we have $U^{\varepsilon}_t\in\partial\varphi(J^{\varepsilon}(Y_t^{\varepsilon}))$ and $V^{\varepsilon}_t\in\partial\psi(\bar{J}^{\varepsilon}(Y_t^{\varepsilon}))$ which imply
\begin{eqnarray*} 
\begin{array}{rrr}
\langle U^{\varepsilon}_t,\,u_t-J^{\varepsilon}(Y_t^{\varepsilon})\rangle+\varphi(J^{\varepsilon}(Y_t^{\varepsilon}))\leq  \varphi(u_t),\;\;  d\P\otimes dt\mbox{-a.e}&\\
\mbox{and}&\\
\langle V^{\varepsilon}_t,\,v_t-J^{\varepsilon}(Y_t^{\varepsilon})\rangle+\psi(\bar{J}^{\varepsilon}(Y_t^{\varepsilon}))\leq  \psi(v_t)\;\;  d\P\otimes dA_t\mbox{-a.e}.& 
\end{array}
\end{eqnarray*}
Therefore we obtain
\begin{eqnarray}\label{ii3} 
\begin{array}{rrr}
\displaystyle \int_a^b \langle U^{\varepsilon}_t,\, u_t-J^{\varepsilon}(Y_t^{\varepsilon})\rangle  dt+\int_a^b\varphi(J^{\varepsilon}(Y_t^{\varepsilon}))dt\leq  \int_a^b\varphi(u_t)dt&\\
\mbox{and}&\\
\displaystyle \int_a^b \langle V^{\varepsilon}_t,\,v_t-J^{\varepsilon}(Y_t^{\varepsilon})\rangle dA_t+\int_a^b\psi(\bar{J}^{\varepsilon}(Y_t^{\varepsilon})) dA_t\leq  \int_a^b\psi(v_t)dA_t.&
\end{array}
\end{eqnarray}
Taking the $\liminf$ in \eqref{ii3}, we have
\begin{eqnarray*}
\begin{array}{rrr}
\displaystyle \int_a^b U_t(u_t-Y_t)dt+\int_a^b\varphi(Y_t)dt\leq  \int_a^b\varphi(u_t)dt&\\
\mbox{and}&\\
\displaystyle \int_a^b V_t(v_t-Y_t) dA_t+\int_a^b\psi(Y_t) dA_t\leq  \int_a^b\psi(v_t)dA_t& 
\end{array}
\end{eqnarray*}
in probability. We have used \eqref{ii2}, \eqref{ii3} and the fact that $\varphi$ and $\psi$ are lower semi continuous. As $a, b, u$ and $v$ be taken arbitrary we get
\begin{eqnarray*}
\begin{array}{rrr}
\langle U_t,\,u_t-Y_t)+\varphi(Y_t)\rangle \leq \varphi(u_t), & d\P\otimes dt\mbox{-a.e}\\
\mbox{and}&\\
\langle V_t,\,v_t-Y_t)\rangle+\psi(Y_t\rangle\leq  \psi(v_t),&  d\P\otimes dt\mbox{-a.e},
\end{array}
\end{eqnarray*}
which implies $(ii)$.\\
Finally taking limit in \eqref{eq1} yields $(iii)$ and ended the step 4 and also the proof of existence.\\
{(\bf ii)} {\bf Uniqueness}\\
Let $(Y,Z,U,V)$ and $(\tilde{Y},\tilde{Z},\tilde{U},\tilde{V})$ be two solutions of MGBDSDE \eqref{A}. In light of the computation used in Step 3, we get for all $t\in [0,T]$,
\begin{eqnarray}\label{equniq}
\E\left[\sup_{t\leq s\leq T}|Y_s-\tilde{Y}_s|^2+\int_{t}^{T}\|Z_s-\tilde{Z}_s\|ds\right]\leq C\int_{t}^{T}\rho(\E(\sup_{s\leq u\leq T}Y_u-\tilde{Y}_u|^{2}))ds.
\end{eqnarray} 
Next,  Bihari inequalities yields
\begin{eqnarray*}
\E\left[\sup_{t\leq s\leq T}|Y_s-\tilde{Y}_s|^2\right]=0,\;\; t\in [0,T]. 
\end{eqnarray*}
Particularly setting $t=0$, we obtain
\begin{eqnarray*}
\E\left[\sup_{0\leq s\leq T}|Y_s-\tilde{Y}_s|^2\right]=0. 
\end{eqnarray*}
In view of \eqref{equniq}, and since $\rho(0)=0$, we obtain
\begin{eqnarray*}
\E\left[\int_{t}^{T}\|Z_s-\tilde{Z}_s\|ds\right]=0, \;\; t\in [0,T].
\end{eqnarray*}
Particularly setting $t=0$, we obtain
\begin{eqnarray*}
\E\left[\int_{0}^{T}\|Z_s-\tilde{Z}_s\|ds\right]=0. 
\end{eqnarray*}
Finally we have $Y=\tilde{Y}$ and $Z=\tilde{Z}$.

\end{proof}

\subsection{Comparison principle for variational GBDSDEs under non-Lipschitz condition}
In this subsection, we only consider one-dimensional variational generalized BDSDEs, i.e., $k=1$. We consider the following variational GBDSDEs: $0\leq t \leq T$ for $i=1,\;2$
\begin{eqnarray}\label{C1}
Y^i_t+\int_t^T U^i_sds+\int_t^T V^i_s dA_s &=&\xi^i+\int_t^Tf^i(s,Y^i_s,Z^i_s)ds+\int_t^Tg^i(s,Y^i_s)dA_s\nonumber\\
&&+\int_t^T h(s,Y^i_s,Z^i_s)\overleftarrow{dB}_s-\int_t^T Z^i_s dW_s,
\end{eqnarray}
where $(Y^{i}_t,U^{i}_t)\in \partial \varphi$ and $(Y^i_t,V^i_t)\in \partial \psi$. 
\begin{remark}
Since for $i=1,2$, $(Y^{i}_t,U^{i}_t)\in \partial \varphi$ and $(Y^i_t,V^i_t)\in \partial \psi$, we have 
\begin{eqnarray}\label{C11}
(Y^2-Y^1)(U^2-U^1)\geq 0\;\;\:\mbox{and}\;\;\; (Y^2-Y^1)(V^2-V^1)\geq 0.
\end{eqnarray}
\end{remark}
 
If for $f^i,\; g^i,\;i=1,\,2$ and $h$ satisfy the conditions of Theorem \ref{thm3.1}, then there exist a unique quadruplet of measurable processes $(Y^i,U^i,V^i,Z^i)_{i=1,\, 2}$ solution of \eqref{C1}. Assume 
\begin{description}
\item [(H5)]
\begin{itemize}
\item [$(i)$] $\xi^1\leq \xi^2$, a.s., 
\item [$(ii)$]$f^1(t,Y^1_t,Z^1_t)\leq f^2(t,Y^1_t,Z^1_t)$ a.s. and $g^1(t,Y^1_t)\leq g^2(t,Y^1_t),\; d\P\otimes dA_t$-a.e. $t\in[0,T]$,
or 
\item [] $f^1(t,Y^2_t,Z^2_t)\leq f^2(t,Y^2_t,Z^2_t)$ a.s. and $g^1(t,Y^2_t)\leq g^2(t,Y^2_t),\; d\P\otimes dA_t$-a.e. 
\end{itemize} 
\end{description}
We have the following comparison result.
\begin{theorem}
Assume the conditions of Theorem \ref{thm3.1}. Let for $i=1,2,\,(Y^i,U^i,V^i,Z^i)$ be solutions of GBDSDE \eqref{C1}. If $({\bf H5})$ holds, then $Y^1_t\leq Y^2_t$ a.s. $\forall\; t\in[0,T]$. 
\end{theorem}
The proof of this result follows the proof of Theorem 3.1 appear in \cite{Shial}.
\begin{proof}
Let us assume that $f^1(t,Y^2_t,Z^2_t)\leq f^2(t,Y^2_t,Z^2_t)$ and $g^1(t,Y^2_t)\leq g^2(t,Y^2_t)$, a.s., a.e. $t\in[0,T]$. Denoting $\overline{Y}_t=Y^2_t-Y^1_t,\; \overline{Z}_t=Z^2_t-Z^1_t,\; \overline{U}_t=U^2_t-U^1_t,\; \overline{V}_t=V^2_t-V^1_t$, it follows from  Itô formula to $e^{\mu t}|\overline{Y}^-_t|^2$ that
\begin{eqnarray}\label{C2}
&&|e^{\mu t}\overline{Y}^-_t|^2+\mu\int_t^Te^{\mu s}|\overline{Y}^-_s|^2ds-\int^T_te^{\mu s}\overline{Y}^-_s\overline{U}_sds-\int^T_te^{\mu s}\overline{Y}^-_s\overline{V}_sdA_s\nonumber\\
&=&|e^{\mu T}(\xi^2-\xi^1)^{-}|^2\nonumber\\
&&-\int^T_te^{\mu s}\overline{Y}^-_s(f^2(s,Y^2_s,Z^2_s)-f^1(s,Y^1_s,Z^1_s))ds-\int^T_te^{\mu s}\overline{Y}^-_s(g^2(s,Y^2_s)-g^1(s,Y^1_s))dA_s\nonumber\\
&&-\int^T_te^{\mu s}\overline{Y}^-_s(h(s,Y^2_s,Z^2_s)-h(s,Y^1_s,Z^1_s))\overleftarrow{dB}_s+ \int^T_te^{\mu s}\overline{Y}^-_s\overline{Z}_sdW_s\nonumber\\
&&+\int^T_t{\bf 1}_{\{\overline{Y}_s<0\}}e^{\mu s}|h(s,Y^2_s,Z^2_s)-h(s,Y^1_s,Z^1_s)|^2ds- \int^T_t{\bf 1}_{\{\overline{Y}<0\}}e^{\mu s}|\overline{Z}_s|^2ds\nonumber\\
&\leq &|e^{\mu T}(\xi^2-\xi^1)^{-}|^2\nonumber\\
&&+\int^T_te^{\mu s}\overline{Y}^-_s(f^1(s,Y^1_s,Z^1_s)-f^1(s,Y^2_s,Z^2_s))ds-\int^T_te^{\mu s}\overline{Y}^-_s(g^1(s,Y^2_s)-g^1(s,Y^1_s))dA_s\nonumber\\
&&+\int^T_te^{\mu s}\overline{Y}^-_s(h(s,Y^1_s,Z^1_s)-h(s,Y^2_s,Z^2_s))\overleftarrow{dB}_s+ \int^T_te^{\mu s}\overline{Y}^-_s\overline{Z}_sdW_s-\int^T_te^{\mu s+}\overline{Y}^-_sdL_s\nonumber\\
&&+\int^T_t{\bf 1}_{\{\overline{Y}_s<0\}}e^{\mu s}|h(s,Y^2_s,Z^2_s)-h(s,Y^1_s,Z^1_s)|^2ds- \int^T_t{\bf 1}_{\{\overline{Y}<0\}}e^{\mu s}|\overline{Z}_s|^2ds.
\end{eqnarray}
The last inequality is obtained thanks to 
\begin{eqnarray*}
f^2(s,Y^2_s,Z^2_s)-f^1(s,Y^2_s,Z^2_s)\geq 0
\\
g^2(s,Y^2_s)-f^1(s,Y^2_s)\geq 0,	
\end{eqnarray*}
which follows from $({\bf H5})$. Recall again $({\bf H4})$, we have $e^{\mu s}(\xi^2-\xi^1)\geq 0$ so that  
 \begin{eqnarray}
 	\E(|e^{\mu s}(\xi^2-\xi^1)^-|^2)=0.\label{C3}
 \end{eqnarray}
 Since $(Y^{i},Z^{i}),\; i=1,2$ are in $\mathcal{S}^{2}(\R)\times\mathbb{M}^{2}(\R^{d})$, we have 
 \begin{eqnarray}\label{C4}
 \E\left(\int_t^Te^{\mu s}\overline{Y}^-_s\overline{Z}_sdW_s\right)=0.
 \end{eqnarray}
 and 
 \begin{eqnarray}\label{C5}
 \E(\left(\int^T_te^{\mu s}\overline{Y}^-_s(h(s,Y^2_s,Z^2_s)-h(s,Y^1_s,Z^1_s))\overleftarrow{dB}_s\right)=0.
 \end{eqnarray}
 By Assumptions $({\bf H2})$  and the basic inequality $2ab\leq \delta a^2+\frac{1}{\delta}b^2$ we get
 \begin{eqnarray}\label{C6}
 \int^T_te^{\mu s}\overline{Y}^-_s(f^1(s,Y^1_s,Z^1_s)-f^1(s,Y^2_s,Z^2_s))ds &\leq & \delta\int_t^Te^{\mu s}|\overline{Y}^-_s|^2ds+\frac{1}{\delta}\int_t^Te^{\mu s}\rho(|\overline{Y}_s|^2){\bf 1}_{\{\overline{Y_s}<0\}}ds\nonumber\\
 &&+\frac{K}{\delta}\int_t^Te^{\mu s}|\overline{Z}_s|^2{\bf 1}_{\{\overline{Y_s}<0\}}ds,
 \end{eqnarray}
 
 \begin{eqnarray}\label{C7}
 -\int^T_te^{\mu s}\overline{Y}^-_s(g^1(s,Y^2_s)-g^1(s,Y^1_s))dA_s&\leq & \beta\int_t^Te^{\mu s}|\overline{Y}_s^-|^2{\bf 1}_{\{\overline{Y_s}<0\}}dA_s
 \end{eqnarray}
 and 
 \begin{eqnarray}\label{C8}
 &&\int^T_t{\bf 1}_{\{\overline{Y}_s<0\}}e^{\mu s}|h(s,Y^2_s,Z^2_s)-h(s,Y^1_s,Z^1_s)|^2ds\nonumber\\ &\leq & \int_t^Te^{\mu s}\rho(|\overline{Y}_s|^2){\bf 1}_{\{\overline{Y_s}<0\}}ds+\alpha\int_t^Te^{\mu s}|\overline{Z}_s|^2{\bf 1}_{\{\overline{Y_s}<0\}}ds.
 \end{eqnarray}
Putting \eqref{C3}-\eqref{C8} in \eqref{C2} we get
\begin{eqnarray*}
&&\E(e^{\mu t}|\overline{Y}_t^-|^2)+(\mu-\beta-\delta)\E\left(\int_t^Te^{\mu s}|\overline{Y}^-_s|^2ds\right)+\left(1-\alpha-\frac{K}{\delta}\right)\E\left(\int_t^Te^{\mu s}|\overline{Z}_s|^2{\bf 1}_{\{\overline{Y_s}<0\}}ds\right)\nonumber\\
&\leq &
\left(\frac{1}{\delta}+1\right)\E\left(\int^T_te^{\mu s}\rho(|\overline{Y}_s^-|^2)ds\right).
\end{eqnarray*}
Finally choosing $\mu>0$ and $\delta>0$ such that $\mu-\beta-\delta>0$ and $1-\alpha-\frac{K}{\delta}>0$, we have 
\begin{eqnarray*}
\E(e^{\mu t}|\overline{Y}_t^-|^2)&\leq &
C\E\left(\int^T_te^{\mu s}\rho(|\overline{Y}_s^-|^2)ds\right),
\end{eqnarray*}
which by using Fubini's theorem and Jensen's inequality leads to
\begin{eqnarray*}
\E(|\overline{Y}_t^-|^2)&\leq &
C\int^T_t\rho(\E(|\overline{Y}_s^-|^2))ds.
\end{eqnarray*}
Then we can use Bihari's inequality to obtain $\E(|\overline{Y}^-_t|^2)=0,\;\; \forall\, t\in[0,T]$, and so $Y^1_t\leq Y^2_t$, a. s., for all $t\in [0,T]$.
\end{proof}

\section{Stochastic viscosity solutions of variational SPDEs with a nonlinear Neumann-Dirichlet boundary condition}
In this section, we defined the notion of stochastic viscosity solution for variational SPDE \eqref{i1}. Next, via variational GBDSDEs studied in the previous section, we give  a probabilistic representation of this variational SPDE in a such stochastic viscosity sense.

\subsection{ Preliminaries and definitions}

Let ${\bf F}^{B}=\{\mathcal{F}_{t,T}^{B}\}_{0\leq t\leq T}$ be the backward filtration defined in Section 2. We also denote by
${\mathcal{M}}^{B}_{0,T}$ all ${\bf F}^{B}$-stopping
times $\tau$ such $0\leq \tau\leq T$, a.s.  For generic Euclidean spaces $E$ and
$E_{1}$, we introduce the following spaces:
\begin{enumerate}
\item $\mathcal{C}^{k,n}([0,T]\times
E; E_{1})$ denotes the space of all $E_{1}$-valued functions
defined on $[0,T]\times E$ which are $k$-times continuously
differentiable in $t$ and $n$-times continuously differentiable in
$x$; and $\mathcal{C}^{k,n}_{b}([0,T]\times E; E_{1})$ denotes the
subspace of $\mathcal{C}^{k,n}([0,T]\times E; E_{1})$ in which all  function have uniformly bounded partial derivatives.
\item $\mathcal{C}^{k,n}({\bf F}^{B},[0,T]\times E; E_{1})$
(resp.$\mathcal{C}^{k,n}_{b}({\bf F}^{B},[0,T]\times E; E_{1})$) is
the space of all random fields $\gamma\in
\mathcal{C}^{k,n}([0,T]\times E; E_{1})$ (resp.
$\mathcal{C}^{k,n}([0,T]\times E; E_{1})$), such
that for fixed $x\in E$, \newline the mapping
$\displaystyle{(t,\omega)\rightarrow \gamma(t,\omega,x)}$ is
${\bf F}^{B}$-progressively measurable.
\item For a real number $ p\geq 1$, let
$L^{p}(\mathcal{F}_T;E)$ be a set of all $E$-valued,
$\mathcal{F}_T$-measurable random variable $\xi$ such that $
\E|\xi|^{p}<+\infty$.
\end{enumerate}
For $(t,x,y)\in[0,T]\times\R^{d}\times\R^k$, we denote
$D_{x}=(\frac{\partial}{\partial
x_{1}},....,\frac{\partial}{\partial x_{d}}),\,
D_{xx}=(\partial^{2}_{x_{i}x_{j}})_{i,j=1}^{d}$,
$D_{y}=\frac{\partial}{\partial y}, \,\
D_{t}=\frac{\partial}{\partial t}$.  Let $\Theta$ and $\Gamma$ be  defined respectively by
\begin{eqnarray*}
\Theta=\{x\in \R^d:\;\;\phi(x)>0\},\;\;\; Bd(\Theta)=\{x\in\R^d:\;\;\phi(x)=0\},\; \mbox{for any}\; \phi\in\mathcal{C}^{2}_b(\R^{d})
\end{eqnarray*}
and 
\begin{eqnarray*}
\Gamma u(t,x)=\sum_{j=1}^{d}\frac{\partial\phi}{\partial x_j}(x)\frac{\partial u}{\partial x_j}(t,x)=\langle \nabla\phi(x),\nabla u(t,x)\rangle,\, \mbox{for all}\; x\in Bd(\Theta)
\end{eqnarray*}.

For example let $\phi:\R^d\to \R$ defined as follows: for $x=(x_1,\cdot,\cdot\cdot,x_d),\;\; \phi(x)=x_1$. Hence
\begin{eqnarray*}
\Theta=\{x\in \R^d:\;\;x_1>0\},\;\;\; Bd(\Theta)=\{x\in\R^d:x_1=0\}\simeq \R^{d-1}.
\end{eqnarray*} 
Moreover, 
\begin{eqnarray*}
\Gamma u(t,x)=\frac{\partial u}{\partial x_1}(t,x)=\langle e_1,\nabla u(t,x)\rangle,
\end{eqnarray*}
where $e_1=(1,0,\cdot\cdot\cdot,0)$.

We work in this section with the following assumptions:
\begin{description}
\item [(H6)] $f, g$ are functions defined from $\Omega\times[0,T]\times\overline{\Theta}\times\R$ to $\R$ such that for all $t\in[0,T],\,  (x,y)\mapsto  f(t,x,y)$ is $K$-Lipschitz and  $g$ is $K$-Lipschitz continuous in $t,x,y$.
\item [(H7)] The functions $\sigma : \R^d \to \R^{d\times d}$ and $b : \R^d \to \R^{d}$ are bounded and $K$-Lipschitz continuous.
\item [(H8)] The function $\chi: R^d\to  \R$ is continuous, such that for some constants $K, \ p > 0$,
\begin{eqnarray*}
|\chi(x)| + |\varphi(\chi(x))| + |\psi(\chi(x))|\leq K(1 + |x|^p).
\end{eqnarray*}
\item[(H9)] $h\in{\mathcal{C}}_{b}^{0,2,3}([0,T]\times\overline{\Theta}\times\R;\R^{l})$.
\end{description}
\begin{remark}
The Lipschitz condition on $g$ which respect time variable is necessary in order to obtain the continuity of the function $(t,x)\mapsto Y^{t,x}$ where $(t,x)$ is initial data of the forward SDE \eqref{eqmarkov1} (i.e $X^{t,x}_s=x$ for all $s\in[0,t]$).
\end{remark}

\subsection{Notion of stochastic viscosity solution of variational SPDEs with a nonlinear Neumann-Dirichlet boundary condition}

The aim of this subsection is to set up the notion of stochastic viscosity solution of variationalstochastic partial differential equations (VSPDEs, in short) with nonlinear
Neumann-Dirichlet boundary condition:
\begin{eqnarray}
\left\{
\begin{array}{l}
(i)\;\left(\frac{\partial u}{\partial t}(t,x)+
Lu(t,x)+f(t,x,u(t,x))+h(t,x,u(t,x))\frac{\partial\overleftarrow{B_t}}{\partial t}\right)\in\partial\varphi(u(t,x)),\,
(t,x)\in[0,T]\times\Theta,\\\\
(ii)\;\left(\frac{\partial u}{\partial
n}(t,x)+g(t,x,u(t,x))\right)\in\partial\psi(u(t,x)),\,(t,x)\in[0,T]\times Bd(\Theta), \\\\
(iii)\; u(T,x)=\chi(x),\,x\in\overline{\Theta}.
\end{array}\right.\label{SPVI}
\end{eqnarray}
Our method is purely probabilistic and follows the idea in \cite{BMa}. To this fact, let us recall their statement.
\begin{definition}
Let $\tau\in {\mathcal{M}}^{B}_{0,T}$, and
$\xi\in\mathcal{F}_{\tau,T}$. We say that a sequence of random
variables $(\tau_k,\xi_k)$ is a $(\tau,\xi)$-approximating sequence
if for all $k$, $(\tau_k,\xi_k)\in{\mathcal{M}}^{B}_{\infty}\times
L^{2}(\mathcal{F}_{\tau,T},\Theta)$ such that
\begin{itemize}
\item [(i)] $\xi_k\rightarrow\xi$  in probability;
\item [(ii)] either $\tau_k\uparrow\tau$ a.s., and $\tau_k<\tau$ on the set $\{\tau>0\}$; or $\tau_k\downarrow\tau$ a.s., and $\tau_k>\tau$ on the set $\{\tau<T\}$.
\end{itemize}
\end{definition}
\begin{definition}
Let $(\tau,\xi)\in \mathcal{M}_{0,T}^B\times
L^2\left(\mathcal{F}^{B}_{\tau,T}; \Theta\right)$ and $u\in
\mathcal{C}\left(\mathbf{F}^B, [0,T]\times
\overline{\Theta}\right)$. A triplet of $(a, p,X)$ is called a
stochastic $h$-superjet of $u$ at $(\tau,\xi)$ if $(a,p,X)$ is an $\R\times\R^d\times\mathcal{S}(d)$-valued, $\mathcal{F}_{\tau,T}^B$-measurable random vector,  such that setting 
$b=h(\tau,\xi,u(\tau,\xi)),\;c=(h\partial_uh)(\tau,\xi,u(\tau,\xi))$ and \newline $q=\partial_xh(\tau,\xi,u(\tau,\xi))+\partial_uh(\tau,\xi,u(\tau,\xi))p$, and for all $(\tau, \xi)$-approximating sequence $(\tau_k,\xi_k)$, we have
\begin{eqnarray}
u(\tau_k,\xi_k)&\leq& u(\tau,\xi)+a(\tau_k-\tau)+b(B_{\tau_k}-B_{\tau})+\frac{c}{2}(B_{\tau_k}-B_{\tau})^2+\langle p,\xi_k-\xi\rangle\nonumber\\
&&+\langle q,\xi_k-\xi\rangle(B_{\tau_k}-B_{\tau})+\frac{1}{2}\langle X(\xi_k-\xi),\xi_k-\xi\rangle\nonumber\\
&&+o(|\tau_k-\tau|)+o(|\xi_k-\xi|^2).\label{jet1}
\end{eqnarray}
We denote by $\mathcal{J}^{1,2,+}_{h} u(\tau,\xi)$ the set of all
stochastic $h$-superjet of $u$ at $(\tau, \xi)$. Similarly, the
triplet of $(a, p,X)$ is a stochastic $h$-subjet of $u$ at $(\tau,
\xi)$ if the inequality in \eqref{jet1} is
reversed and $\mathcal{J}^{1,2,-}_{h} u(\tau,\xi)$ denotes
the set of all stochastic $h$-subjet of $u$ at $(\tau,\xi)$.
\end{definition}
We define the stochastic viscosity solution of MSPDE \eqref{SPVI}. To simplify let us set
\begin{eqnarray*}
V_{f}(\tau,\xi,a,p,X)=-a-\frac{1}{2}{\rm Trace}(\sigma\sigma^*(\xi)X)-\langle
p,b(\xi)\rangle-f\left(\tau,\xi,u(\tau,\xi)\right).
\end{eqnarray*}
\begin{definition}\label{defvisco}
Let $u \in \mathcal{C}\left(\mathbf{F}^B, [0,T]\times
\overline{\Theta}\right)$ satisfying $u\left(T,x\right)=\chi\left(x\right)$, for all $x\in
\overline{\Theta}$. Moreover, $\forall\; (\tau,\xi)\in\mathcal{M}_{0,T}^B\times L^2\left(\mathcal{F}^{B}_{\tau,T};\overline{\Theta}\right)$
\begin{eqnarray}\label{ViscM1}
u(\tau,\xi)\in {\rm Dom}(\varphi),\;\; \mbox{on}\; \; \left\{\xi\in
\Theta\right\}
\end{eqnarray}
and
\begin{eqnarray}\label{ViscM2}
u(\tau,\xi)\in  {\rm Dom}(\psi),\;\; \mbox{on}\; \; \left\{\xi\in
Bd(\Theta)\right\}.
\end{eqnarray}
 The function $u$ is called a stochastic viscosity subsolution of MSPDE \eqref{SPVI} if at any $(\tau,\xi)\in\mathcal{M}_{0,T}^B\times L^2\left(\mathcal{F}^{B}_{\tau,T};\overline{\Theta}\right)$, for any $(a, p,X)\in\mathcal{J}^{1,2,+}_{h} u(\tau,\xi)$, we have $\P$-a.s.
\begin{itemize}
\item[(a)] on $\left\{0<\tau<T\right\}\cap\left\{\xi\in
\Theta\right\}$
\begin{eqnarray}\label{E:def1}
V_{f}(\tau,\xi,a,p,X)+\varphi'_l(u(\tau,\xi))-\frac{1}{2}(h\partial_uh)(\tau,\xi,u(\tau,\xi))\leq 0;
\end{eqnarray}
\item[(b)] on $\left\{0<\tau<T\right\}\cap\left\{\xi\in
Bd(\Theta)\right\}$
\begin{eqnarray}
&&\min\left(V_{f}(\tau,\xi,a,p,X)+\varphi'_l(u(\tau,\xi))-\frac{1}{2}(h\partial_uh)(\tau,\xi,u(\tau,\xi))\right.,\nonumber\\
&&\left.\langle\nabla
\phi(\xi),p\rangle+\psi'_l(u(\tau,\xi))-g(\tau,\xi,u(\tau,\xi))\right)\leq 0. 
\label{E:viscosity01}
\end{eqnarray}
\end{itemize}
The function $u$ is called a stochastic viscosity supersolution of MSPDE \eqref{SPVI} if at any $(\tau,\xi)\in\mathcal{M}_{0,T}^B\times L^2\left(\mathcal{F}^{B}_{\tau,T};\overline{\Theta}\right)$, for any $(a, p,X)\in\mathcal{J}^{1,2,-}_{h} u(\tau,\xi)$, it hold $\P$-a.s.
\begin{itemize}
\item[(a)] on $\left\{0<\tau<T\right\}\cap\left\{\xi\in
\Theta\right\}$
\begin{equation}\label{E:def01}
V_f(\tau,\xi,a,p,X)+\varphi'_r(u(\tau,\xi))-\frac{1}{2}(h\partial_uh)(\tau,\xi,u(\tau,\xi))\geq 0;
\end{equation}
\item[(b)] on $\left\{0<\tau<T\right\}\cap\left\{\xi\in
Bd(\Theta)\right\}$
\begin{eqnarray}
&&\max\left(V_f(\tau,\xi,a,p,X)+\varphi'_r(u(\tau,\xi))-\frac{1}{2}(h\partial_uh)(\tau,\xi,u(\tau,\xi)),\right.\nonumber\\
&&\left.\langle\nabla \phi(\xi),p\rangle+\psi'_l(u(\tau,\xi))-g(\tau,\xi,u(\tau,\xi))\right)\geq 0  \label{E:viscosity001}
\end{eqnarray}
\end{itemize}
Finally, a random field $u \in \mathcal{C}\left(\mathbf{F}^B,
[0,T]\times \overline{\Theta}\right)$ is called a stochastic
viscosity solution of variational SPDE \eqref{SPVI} if it is both a stochastic viscosity subsolution and a
stochastic viscosity supersolution.
\end{definition}
\begin{remark}
We observe that if $f$ and $g$ are deterministic and $h\equiv 0$,
Definition \ref{defvisco} coincides with the definition of
(deterministic) viscosity solution of MPDE given by Maticiuc and R\u{a}\c{s}canu in \cite{MR}.
\end{remark}
We now state the notion of random viscosity solution which will be a bridge
link to the stochastic viscosity solution and its deterministic counterpart.
\begin{definition}
A random field $u\in C({\bf F}^B, [0,T]\times\overline{\Theta})$ is called an $\omega$-wise viscosity solution if
for $\P$-almost all $\omega\in \Omega,\;  u(\omega,\cdot,\cdot)$ is a (deterministic) viscosity solution of MSPDE associated to data $(f,g,0,\chi,\varphi,\psi)$.
\end{definition}

\subsection{Doss-Sussmann transformation}
In this subsection, using the Doss-Sussman transformation, our goal is to establish the link between the notion of stochastic viscosity solution for variational SPDE \eqref{SPVI} and the notion of viscosity solution for PDE with random coefficient. For this fact, let introduce the process $\eta\in C({\bf F}^B, [0,T]\times\R^n\times\R)$, as the unique solution of the stochastic differential equation written in Stratonovich form
\begin{eqnarray}
\eta(t,x,y)&=&y+\int_t^T\langle h(s,x,\eta(s,x,y)),
\circ \overleftarrow{dB}_{s}\rangle.\label{p1}
\end{eqnarray}
Note that due to the direction of Itô
integral, \eqref{p1} should be viewed as going from $T$ to $t$ (i.e
$y$ should be understood as the initial value). Under the assumption
$(\bf H9)$, the mapping $y\mapsto\eta(t,x,y)$ is a diffeomorphism for all $(t,x),\; \P$-a.s. (see Protter \cite{Pr}). Hence if we denote its $y$-inverse
  by $\varepsilon$, we get 
\begin{eqnarray*}
\varepsilon(t,x,y)=y-\int_t^TD_y\varepsilon(s,x,y)\langle h(s,x,\eta(s,x,y)),\, 
\circ \overleftarrow{dB}_{s}\rangle.
\end{eqnarray*}
Let us recall the following important proposition in
\cite{BM} (see Lemma 4.8).
\begin{proposition}\label{PropDoss}
Assume $(\bf H6)$-$(\bf H9)$ hold. Let
$(\tau,\xi)\in \mathcal{M}_{0,T}^B\times
L^2\left(\mathcal{F}^{B}_{\tau,T}; \overline{\Theta}\right)$, $u\in
\mathcal{C}\left(\mathbf{F}^B, [0,T]\times\overline{\Theta}\right)$
and $(a_u,X_u,p_u)\in\mathcal{J}^{1,2,+}_{h} u(\tau,\xi)$. Define
$v(\cdot,\cdot) = \varepsilon(\cdot,\cdot, u(\cdot,\cdot))$. Then,
for any $(\tau,\xi)$-approximating sequence $(\tau_k,\xi_k)$, and
for $\P$-a.e., it holds that
\begin{eqnarray*}
v(\tau_k,\xi_k)&\leq& v(\tau,\xi)+a_v(\tau_k-\tau)+b_v(B_{\tau_k}-B_{\tau})+\langle p_v,\xi_k-\xi\rangle\nonumber\\
&&+\langle q_v,\xi_k-\xi\rangle(B_{\tau_k}-B_{\tau})+\frac{1}{2}\langle X_v(\xi_k-\xi),\xi_k-\xi\rangle\nonumber\\
&&+o(|\tau_k-\tau|)+o(|\xi_k-\xi|^2).\label{jet}
\end{eqnarray*}
where
\begin{eqnarray*}
\left\{
\begin{array}{l}
a_v=D_y\varepsilon(\tau,\xi,u(\tau,\xi))a_u\\\\
p_v=D_y\varepsilon(\tau,\xi,u(\tau,\xi))p_u+D_x\varepsilon(\tau,\xi,u(\tau,\xi))\\\\
X_v=D_y\varepsilon(\tau,\xi,u(\tau,\xi))X_u+2D_{xy}\varepsilon(\tau,\xi,u(\tau,\xi))p^*_u+D_{xx}\varepsilon(\tau,\xi,u(\tau,\xi))
+D_{yy}\varepsilon(\tau,\xi,u(\tau,\xi))p_up_u^*
\end{array}\right.
\end{eqnarray*}
Namely, $(a_v,X_v,p_v)\in\mathcal{J}^{1,2,+}_{0} v(\tau,\xi)$

Conversely, let $(\tau,\xi)\in \mathcal{M}_{0,T}^B\times
L^2\left(\mathcal{F}^{B}_{\tau,T}; \Theta\right)$, $v\in
\mathcal{C}\left(\mathbf{F}^B, [0,T]\times\overline{\Theta}\right)$\newline
and $(a_v,X_v,p_v)\in\mathcal{J}^{1,2,+}_{0} v(\tau,\xi)$. Define
$u(\cdot,\cdot) = \eta(\cdot,\cdot, v(\cdot,\cdot))$. Then, the
triplet $(a_u,X_u,p_u)$ given by
\begin{eqnarray*}
\left\{
\begin{array}{l}
a_u=D_y\eta(\tau,\xi,v(\tau,\xi))a_v\\\\
p_u=D_y\eta(\tau,\xi,v(\tau,\xi))p_v+D_x\eta(\tau,\xi,v(\tau,\xi))\\\\
X_u=D_y\eta(\tau,\xi,v(\tau,\xi))X_v+2D_{xy}\eta(\tau,\xi,v(\tau,\xi))p^*_v+D_{xx}\eta(\tau,\xi,v(\tau,\xi))
+D_{yy}\eta(\tau,\xi,v(\tau,\xi))p_vp_v^*
\end{array}\right.
\end{eqnarray*}
satisfies $(a_u,X_u,p_u)\in\mathcal{J}^{1,2,+}_{h} u(\tau,\xi)$.
\end{proposition} 
One of the key ideas of Buckdahn and Ma is to use the Doss-Sussman transformation to convert a SPDE to a PDE with random coefficients, so that the stochastic viscosity solution can be studied $\omega$-wisely. However, if we apply Doss-Sussman transformation to MSPDE \eqref{SPVI} the resulting equation is not necessarily the multivalued PDE studied by Maticiuc and R\u{a}\c{s}canu in \cite{MR}, because of the presence of the subdifferential term. For this reason we will need
the following version of Doss-Sussman transformation.

\begin{corollary}\label{corollary4.8}
Assume that the assumptions $(\bf H6)$-$(\bf H9)$ hold. Let
$(\tau,\xi)\in \mathcal{M}_{0,T}^B\times
L^2\left(\mathcal{F}^{B}_{\tau,T}; \overline{\Theta}\right)$, $u\in
\mathcal{C}\left(\mathbf{F}^B, [0,T]\times\overline{\Theta}\right)$
and  define $v(\cdot,\cdot) = \varepsilon(\cdot,\cdot,
u(\cdot,\cdot))$.
\begin{description}
\item [(1)] $u$ is a subsolution of VSPDE \eqref{SPVI}  if and only if $v(\cdot,\cdot)$ satisfies
\begin{itemize}
\item[\rm(a)] on the event $\left\{0<\tau<T\right\}\cap\left\{\xi\in
\Theta\right\}$
\begin{eqnarray}\label{E:def2}
V_{\widetilde{f}}(\tau,\xi,a_v,p_v,X_v)+\frac{\varphi'_l(\eta(\tau,\xi,v(\tau,\xi))}{D_y\eta(\tau,\xi,v(\tau,\xi))}\leq 0;
\end{eqnarray}
\item[\rm(b)] on $\left\{0<\tau<T\right\}\cap\left\{\xi\in
Bd(\Theta)\right\}$
\begin{eqnarray}
&&\min\left(V_{\widetilde{f}}(\tau,\xi,a_v,p_v,X_v)+\frac{\varphi'_l
(\eta(\tau,\xi,v(\tau,\xi))}{D_y\eta(\tau,\xi,v(\tau,\xi))}\right.,\nonumber\\
&&\left.\langle\nabla
\phi(\xi),p_v\rangle-\widetilde{g}(\tau,\xi,u(\tau,\xi))+\frac{\psi'_l
(\eta(\tau,\xi,v(\tau,\xi))}{D_y\eta(\tau,\xi,v(\tau,\xi))}\right)\leq
0. \label{E:viscosity02}
\end{eqnarray}
\end{itemize}
\item [(2)] $u$ is a supersolution of VSPDE \eqref{SPVI} if and only if $v(\cdot,\cdot)$
 satisfies
\begin{itemize}
\item[\rm(a)] on $\left\{0<\tau<T\right\}\cap\left\{\xi\in
\Theta\right\}$
\begin{eqnarray}\label{E:def02}
V_{\widetilde{f}}(\tau,\xi,a_v,p_v,X_v)+\frac{\varphi'_r(\eta(\tau,\xi,v(\tau,\xi))}{D_y\eta(\tau,\xi,v(\tau,\xi))}\geq 0;
\end{eqnarray}
\item[\rm(b)] on $\left\{0<\tau<T\right\}\cap\left\{\xi\in
Bd(\Theta)\right\}$

\begin{eqnarray}
&&\max\left(V_{\widetilde{f}}(\tau,\xi,a_v,p_v,X_v)+\frac{\varphi'_r
(\eta(\tau,\xi,v(\tau,\xi))}{D_y\eta(\tau,\xi,v(\tau,\xi))},\right.\nonumber\\
&&\left.\langle\nabla
\phi(\xi),p_v\rangle-\widetilde{g}(\tau,\xi,u(\tau,\xi))
+\frac{\psi'_r(\eta(\tau,\xi,v(\tau,\xi))}{D_y\eta(\tau,\xi,v(\tau,\xi))}\right)\geq
0. \label{E:viscosity002}
\end{eqnarray}
\end{itemize}
\end{description}
\end{corollary}
\begin{proof}
Suppose $u$ a stochastic subsolution of VSPDE \eqref{SPVI}. Hence for $(\tau,\xi)\in \mathcal{M}_{0,T}^B\times
L^2\left(\mathcal{F}^{B}_{\tau,T};\overline{\Theta}\right)$ and $(a_u,
p_u,X_u)\in\mathcal{J}^{1,2,+}_{h}u(\tau,\xi)$, $u(\tau,\xi)\in  {\rm Dom}(\varphi)$ on $\left\{\xi\in\Theta\right\}$ and $u(\tau,\xi)\in {\rm Dom}(\psi)$ on $\left\{\xi\in Bd(\Theta)\right\}$. Moreover,
\begin{itemize}
\item[] on $\left\{0<\tau<T\right\}\cap\left\{\xi\in\Theta\right\}$
\begin{eqnarray}
V_{f}(\tau,\xi,a_u,p_u,X_u)+\varphi'_l(u(\tau,\xi))-\frac{1}{2}(h\partial_uh)(\tau,\xi,u(\tau,\xi))\leq 0;\label{z1}
\end{eqnarray}
and
\item[] on $\left\{0<\tau<T\right\}\cap\left\{\xi\in Bd(\Theta)\right\}$

\begin{eqnarray}\label{z2}
&&\min\left(V_{f}(\tau,\xi,a_u,p_u,X_u)+\varphi'_l(u(\tau,\xi))-\frac{1}{2}(h\partial_u
h)(\tau,\xi,u(\tau,\xi)),\right.\nonumber\\
&&\left.\langle\nabla
\phi(\xi),p_u\rangle+\psi'_l(u(\tau,\xi))-g(\tau,\xi,u(\tau,\xi))\right)\leq 0.
\end{eqnarray}
\end{itemize}
Using Proposition \ref{PropDoss}, there exist $(a_v,
p_v,X_v)\in\mathcal{J}^{1,2,+}_{0}v(\tau,\xi)$ such that 
\begin{eqnarray*}
&&V_{f}(\tau,\xi,a_u,p_u,X_u)\\
&=&D_y\eta(\tau,\xi,v(\tau,\xi))(-a_v- \frac{1}{2}Trace(\sigma\sigma^*(\xi)X_v)-\langle p_v,b(\xi)\rangle)\\
&&-\frac{1}{2}Trace(\sigma\sigma^*(\xi)D_{xx}\eta(\tau,\xi,v(\tau,\xi))\\
&&-\frac{1}{2}D_{yy}\eta(\tau,\xi,v(\tau,\xi))|\sigma^*(\xi)p_v|^2-\langle\sigma^{*}(\xi)D_{xy}\eta(\tau,\xi,v(\tau,\xi)),\sigma^*(\xi)p_v\rangle\\
&&-\langle D_{x}\eta(\tau,\xi,v(\tau,\xi)),b(\xi)\rangle-f(\tau, \xi, \eta(\tau,\xi,v(\tau,\xi)))
\end{eqnarray*}
and 
\begin{eqnarray*}
&&\langle\nabla \phi(\xi),p_u\rangle -g(\tau,\xi,\eta(\tau,\xi,v(\tau,\xi)))\\
&=&D_{y}\eta(\tau,\xi,v(\tau,\xi))\langle\nabla \phi(\xi),p_v\rangle
+\langle\nabla \phi(\xi),D_{x}\eta(\tau,\xi,v(\tau,\xi))\rangle-g(\tau,\xi,\eta(\tau,\xi,v(\tau,\xi))).
\end{eqnarray*}
Setting 
\begin{eqnarray*}
\widetilde{f}(t,x,y)&=&\frac{1}{D_y\eta(t,x,y)}\Bigg[f(t,x,\eta(t,x,y))-\frac{1}{2}(h\partial_u h)(t,x,\eta(t,x,y))+L_x\eta(t,x,y)\Bigg].
\end{eqnarray*}
and
\begin{eqnarray*}
\widetilde{g}(t,x,y)=\frac{1}{D_{y}\eta(t,x,y)}(g(t,x,y)-\langle\nabla\phi(\xi),D_{x}\eta(t,x,y)\rangle),
\end{eqnarray*}
it follows from \eqref{z1} and \eqref{z2} that
\begin{description}
\item on $\left\{0<\tau<T\right\}\cap\left\{\xi\in\Theta\right\}$,
\begin{eqnarray*}
V_{\widetilde{f}}(\tau,\xi,a_v,p_v,X_v)+\frac{\varphi'_l(u(\tau,\xi)}{D_y\eta(\tau,\xi,v(\tau,\xi))}\leq 0,
\end{eqnarray*}
\item on $\left\{0<\tau<T\right\}\cap\left\{\xi\in Bd(\Theta)\right\}$
\begin{eqnarray*}
&&\min\left(V_{\widetilde{f}}(\tau,\xi,a_v,p_v,X_v)+\frac{\varphi'_l(\eta(\tau,\xi,v(\tau,\xi)))}
{D_y\eta(\tau,\xi,v(\tau,\xi))}\right.
,\\
&&\left.\langle \nabla\phi(\xi),\,p_v\rangle-\widetilde{g}(\tau,\xi,v(\tau,\xi))+\frac{\psi'_l(\eta(\tau,\xi,v(\tau,\xi)))}{D_y\eta(\tau,\xi,v(\tau,\xi))}\right)\leq 0.
\end{eqnarray*}
\end{description}
The converse part of $(1)$ can be proved similarly. In the same manner one can show the second assertion $(2)$.
\end{proof}
\begin{example}
Let us consider a special case of $h$ by $h(s,x,y)= \beta y$, one get $\eta(t,x,y)=y\exp\left(\beta (B_T-B_t)\right)$ and $\varepsilon(t,x,y)=y\exp\left(-\beta(B_T-B_t)\right)$. In this context, 
\begin{eqnarray*}
v(t,x)=\varepsilon(t,x,u(t,x))=\exp\left(-\beta(B_T-B_t)\right)u(t,x),\; \forall\; (t,x)\in [0,T]\times\overline{\Theta}.
\end{eqnarray*}
On the other hand, for $(\tau,\xi)\in \mathcal{M}_{0,T}^B\times
L^2\left(\mathcal{F}^{B}_{\tau,T}\right))$, if $(a,p,A)$ belongs in $\mathcal{J}^{1,2,+}_{h}u(\tau,\xi)$,  we derive that $(\bar{a},\bar{p},\bar{A})$ defined by  
\begin{eqnarray*}
\bar{a} &=&\exp(-\beta(B_T-B_{\tau}))a,\\\\
\bar{p}&=&\exp(-\beta(B_T-B_{\tau}))p +\exp(-\beta(B_T-B_{\tau}))\nabla u(\tau,\xi),\\\\
\bar{A}&=& \exp(\beta (B_T-B_{\tau}))A+\exp(-\beta(B_T-B_{\tau}))D^2u(t,x),
\end{eqnarray*}
belongs to $\mathcal{J}^{1,2,+}_{0}v(\tau,\xi)$. Moreover, 
\begin{eqnarray*}
\widetilde{f}(t,x,y)&=&\exp(-\beta(B_T-B_t))\left[f(t,x,y\exp(\beta(B_T-B_t)))-\frac{1}{2}y\beta^2\exp(\beta(B_T-B_t))\right],
\end{eqnarray*}

\begin{eqnarray*}
\widetilde{g}(t,x,y)&=&\exp(-\beta(B_T-B_t))g(t,x,y\exp(\beta(B_T-B_t))).
\end{eqnarray*}
and inequalities \eqref{E:def2} and \eqref{E:viscosity02} become respectively
\begin{itemize}
\item[\rm(a)] on the event $\left\{0<\tau<T\right\}\cap\left\{\xi\in
\Theta\right\}$
\begin{eqnarray}\label{Z1}
V_{\widetilde{f}}(\tau,\xi,\bar{a},\bar{p},\bar{A})+\exp(-\beta(B_T-B_{\tau}))\varphi'_l(v(\tau,\xi)\exp(\beta(B_T-B_{\tau})))&\leq & 0;
\end{eqnarray}
\item[\rm(b)] on $\left\{0<\tau<T\right\}\cap\left\{\xi\in
Bd(\Theta)\right\}$
\begin{eqnarray}
&&\min\left(V_{\widetilde{f}}(\tau,\xi,\bar{a},\bar{p},\bar{A})+\exp(-\beta(B_T-B_{\tau}))\varphi'_l(v(\tau,\xi)\exp(\beta(B_T-B_{\tau})))\right.,\label{Z2}\\
&&\left.\langle\nabla
\phi(\xi),\bar{p}\rangle-\widetilde{g}(\tau,\xi,v(\tau,\xi))+\exp(-\beta(B_T-B_{\tau}))\psi'_l(v(\tau,\xi)\exp(\beta(B_T-B_{\tau})))\right)\leq
0,\nonumber 
\end{eqnarray}
\end{itemize}
In addition if $\varphi(x)={\bf I}_{[\alpha,\lambda]}$ and $\psi(x)={\bf I}_{[c,d]}(x)$ where
\begin{eqnarray*}
{\bf I }_{[a_1,a_2]}(x)=
\left\{
\begin{array}{ll}
0,&\mbox{if}\; x\in [a_1,a_2],\\\\
+\infty,&\mbox{if}\; x\notin [a_1,a_2]
\end{array}
\right.
\end{eqnarray*}
we derive that $\theta'_l(y)=\theta'_r(y)=0$ if $y\in ]a_1,a_2[$, $\theta'_l(a_1)=-\infty$, $\theta'_r(a_1)=0$, $\theta'_l(a_2)= 0$, $\theta'_r(a_2)=+\infty$. Therefore, inequality \eqref{Z1} and \eqref{Z2} is equivalent to 
\begin{itemize}
\item[\rm(a)] on the event $\left\{0<\tau<T\right\}\cap\left\{\xi\in
\Theta\right\}$
\begin{eqnarray*}\label{Z1}
&&V_{\widetilde{f}}(\tau,\xi,\bar{a},\bar{p},\bar{A},v)\leq  0\;\; \mbox{sur}\; ]\alpha\exp(-\beta(B_T-B_{\tau}),\lambda\exp(-\beta(B_T-B_{\tau})[\nonumber\\
&&V_{\widetilde{f}}(\tau,\xi,\bar{a},\bar{p},\bar{A},\alpha\exp(-\beta(B_T-B_{\tau}))\leq  0\\
&&V_{\widetilde{f}}(\tau,\xi,\bar{a},\bar{p},\bar{A},\lambda\exp(-\beta(B_T-B_{\tau}))\geq 0
\end{eqnarray*}
\item[\rm(b)] on $\left\{0<\tau<T\right\}\cap\left\{\xi\in
Bd(\Theta)\right\}$
\begin{eqnarray*}
&&\min\left(V_{\widetilde{f}}(\tau,\xi,\bar{a},\bar{p},\bar{A},v),\langle\nabla\phi(\xi),\bar{p}\rangle-\widetilde{g}(\tau,\xi,v(\tau,\xi))\right)\leq 0,\;\; \mbox{sur}\;\; ]\alpha\exp(-\beta(B_T-B_{\tau}),\lambda\exp(-\beta(B_T-B_{\tau})[\\
&&\min\left(V_{\widetilde{f}}(\tau,\xi,\bar{a},\bar{p},\bar{A},\alpha\exp(-\beta(B_T-B_{\tau})),\langle\nabla\phi(\xi),\bar{p}\rangle-\widetilde{g}(\tau,\xi,v(\tau,\xi))\right)\leq 0,\\
&&\min\left(V_{\widetilde{f}}(\tau,\xi,\bar{a},\bar{p},\bar{A},\lambda\exp(-\beta(B_T-B_{\tau})),\langle\nabla\phi(\xi),\bar{p}\rangle-\widetilde{g}(\tau,\xi,v(\tau,\xi))\right)\geq 0,
\end{eqnarray*}
\end{itemize}
\end{example}
\subsection{Probabilistic representation result for stochastic viscosity solution to variational SPDE}
The main objective of this subsection is to use the solution of VGBDSDE introduced in Section 2 in the Markovian framework to give a probabilistic representation (in stochastic viscosity sense) for the variational stochastic partial differential equations \eqref{SPVI}.

Roughly speaking, for $(t,x)\in [0,T]\times \overline{\Theta}$, let  consider the forward backward doubly SDE
\begin{eqnarray}\label{eqmarkov1}
\left\{
\begin{array}{l}
\displaystyle X_s^{t,x}=x+\int^{s}_tb(X^{t,x}_r)dr+\int_t^{s}\nabla\phi(X^{t,x}_r)dA_s^{t,x}+\int_t^{s}\sigma(r,X_r^{t,x})dW_r\\
s \mapsto A_s^{t,x} \;\; \mbox{is non-descreasing}\; \mbox{such that}\;
A^{t,x}_s =\int_t^{s\vee t}{\bf 1}_{\{X_r^{t,x}\in Bd(\Theta)\}}dA_r^{t,x},\; \;\; s\in[t,T].\\\\
(Y^{t,x}_s,U^{t,x}_s)\in \partial \varphi, \,\, d\overline{\mathbb{P}}\otimes ds, \ (Y^{t,x}_s,V^{t,x}_s)\in \partial \psi, \ \, d\overline{\mathbb{P}}\otimes dA_s,\\
\displaystyle Y_s^{t,x}+\int_s^TU^{t,x}_rdr+\int_s^TV^{t,x}_rdA^{t,x}_r=\chi(X^{t,x}_T)+\int^T_sf(r,X^{t,x}_r,Y^{t,x}_r,Z^{t,x}_r)dr\\
\displaystyle +\int_s^Tg(r,X^{t,x}_r,Y^{t,x}_r)dA_s^{t,x}+\int_s^T h(r,X_r^{t,x},Y^{t,x}_r,Z^{t,x}_r)\overleftarrow{dB}_s-\int_s^TZ^{t,x}_sdW_r, \;\; s\in[t,T].
\end{array}\right.
\end{eqnarray}
We recall that \eqref{eqmarkov1} admit a unique solution $\{(X^{t,x}_s,\,A^{t,x}_s,Y^{t,x}_s, Z^{t,x}_s, U^{t,x}_s,V^{t,x}_s); \; s\in [t,T]\}$,  (see \cite{LZ} for RSDE (resp. Theorem \ref{thm3.1} for VGBDSDEs)).

As announced in the introduction, the continuity of the different processes with respect to the initial data $(t,x)$ is very essential for the proof our result. The first concerns the forward SDE as follows.
\begin{proposition}\label{P:continuity00}
Assume $(\bf H7)$ holds. Then for any $\kappa,\; p\geq 2$, there exists a constant $ C>0$ such that for all $0\leq t<t'\leq T$
and $x,\,x'\in \overline{\Theta}$ we have,
\begin{description}
\item[(a)]
\begin{eqnarray*}
\mathbb{E}\left[\sup_{0\leq s\leq
T}\left|X^{t,x}_{s}-X^{t',x'}_{s}\right|^p+\sup_{0\leq s\leq
T}\left|A_s^{t,x}-A_{s}^{t',x'}\right|^p\right] \leq
C\left[ |t'-t|^{p/2}+|x-x'|^{p}\right]\label{continuity1}
\end{eqnarray*}
\item[(b)]
\begin{eqnarray*}
\mathbb{E}\left(\left|A_s^{t,x}\right|^p \right) \leq C(1+|(s\vee t)-t|^p).\label{bound1}
\end{eqnarray*}
\item[(c)]
\begin{eqnarray*}
\mathbb{E}\left(e^{\kappa A_s^{t,x}}\right) \leq C.\label{bound2}
\end{eqnarray*}
\item[(d)] $\displaystyle [0,T]\times\overline{\Theta}\ni(t,x)\mapsto \E\left(\int_t^Th_1(s,X_s^{t,x})ds\right)+\E\left(\int_t^Th_2(s,X_s^{t,x})dA^{t,x}_s\right)$ is continuous, for every continuous functions $h_1, \, h_2 : [0, T ] \times\overline{\Theta} \rightarrow\R$.
\end{description}
\end{proposition}
\begin{proof}
This proof follows the similar argument used in \cite{PZ}.  We apply Itô's
formula to the semimartingale
\begin{eqnarray*}
\exp\left[\delta\left(\phi(X^{t,x}_s)+\phi(X^{t',x'}_s)\right)\right]\left|X^{t,x}_s-X^{t',x'}_s\right|^{p},
\end{eqnarray*}
where $\delta$ is a strictly positive constant (which exists due to (\cite{PR}, Theorem 4.47) such
\begin{eqnarray*}
-\langle X^{t,x} _s- X^{t',x'}_s,\,\nabla \phi(X^{t,x}_s)dA^{t,x}_s-\nabla\phi(X^{t',x'}_s)dA^{t',x'}_s\rangle\leq \delta |X^{t,x} _s-X^{t',x'}_s|^2\left(dA^{t,x }_s+dA^{t',x'}_s\right),\;\; \mbox{ a.s.}
\end{eqnarray*} 
Hence by the standard SDE estimates we obtain
\begin{eqnarray*}
\E\left(\sup_{0\leq s\leq T}|X^{t,x}_s-X^{t',x'}_s|^p\right)\leq C\left(|t-t'|^{p/2}+|x-x'|^p+\E\int^{T}_{0}|X^{t,x}_s-X^{t',x'}_s|^pds\right).
\end{eqnarray*}
Next, by Itô formula we have
\begin{eqnarray*}
A^{t,x}_s=\phi(X^{t,x}_s)-\phi(x)-\int_t^{t\vee s}L\phi(X^{t,x}_r)dr-\int_t^{t\vee s}\nabla\phi(X^{t,x}_r)\sigma(X^{t,x}_r)dW_r,
\end{eqnarray*}
where $L$ is is a second order differential operator associated to SDE of \eqref{eqmarkov1} and defined by \eqref{Dop}. 
\end{proof} 
 The second result in this way concerns the continuity of the stochastic process $(Y^{t,x}_t)_{0\leq s \leq T}$ with respect to the initial data $(t,x)$. Since this result has already been established in \cite {PR1} when $h\equiv 0$ and $f$ and $g$ are deterministic, our proof follows their approach. However, because of the presence stochastic integral in our case, we use respectively the Meyer-Zheng and $S$-topology to drive our convergence result.   
\begin{proposition}\label{Prop}
Let $(Y^{t,x}_s, U^{t,x}_s,V^{t,x}_s,
Z^{t,x}_s)_{s\in[0,T]}$ be the unique solution of the VGBDSDE \eqref{eqmarkov1}.
Then, the random field $(t, x)\mapsto Y^{t,x}_t$
is a.s. continuous.
\end{proposition}
\begin{proof}
For an arbitrary $(t,x)\in[0,T]\times \overline{\Theta}$, let us  consider the sequence $(t_n,x_n)_{n\geq 1}$ of $[0,T]\times \overline{\Theta}$ such that $(t_n,x_n)\rightarrow(t,x)$ as $n\rightarrow +\infty$. To prove that $Y^{t_n,x_n}_{t_n}\rightarrow Y^{t,x}_{t}$ a.s as $n\rightarrow +\infty$, it suffice to show that any subsequence of $Y^{t_n,x_n}_{t_n}$ converges to $Y^{t,x}_{t}$. For $(t_{n_k},x_{n_k})$ be an arbitrary subsequence of $(t_n,x_n)$ that we will still note by $(t_n,x_n)$, we set $X^{n}=X^{t_{n},x_{n}},\; A^{n}=A^{t_{n},x_{n}},\;Y^{n}=Y^{t_{n},x_{n}},\; Z^{n}=Z^{t_{n},x_{n}},\; U^{n}=U^{t_{n},x_{n}},\;V^{n}=V^{t_{n},x_{n}}$. It is clear that the processes $(X^n,A^n)$ and $(Y^n,Z^n,U^n,V^n)$ satisfy respectively this two equations : 
\begin{eqnarray}
\left\{
\begin{array}{ll}
\displaystyle X_s^{n}=x_n+\int^{s\vee t_n}_{t_n} b\left(X_r^{n}\right)\,{\rm
d}r+\int^{s\vee t_n}_{t_n}\sigma\left(X_r^{n}\right)\,{\rm
d}{W}_r+\int^{s\vee t_n}_{t_n} \nabla \phi\left(X_r^{n}\right)\,{\rm
d}A_r^{n},\\
s\mapsto\ A^{n}_s\; \mbox{is non decreasing},\\
\displaystyle A^{n}_s=\int^{s\vee t_n}_{t_n}{\bf 1}_{\{X^{n}_s\in \partial\Theta\}}\,{\rm
d}A_r^{n}\quad \forall\, s\in [0,T]
\end{array}\right.
\label{rSDEbis}
\end{eqnarray}
and 
\begin{eqnarray}
Y^{n}_s+\int_s^TU^{n}_r\,{\rm
d}r+\int_s^TV^{n}_r\,{\rm
d}A^{n}_r& =&\chi(X^{n}_T)+\int_s^Tf_n(r,X^{n}_r,Y^{n}_r)\,{\rm
d}r+\int_s^Tg_n(r,X^{n}_r,Y^{n}_r)\,{\rm
d}A^{n}_r\nonumber\\
&&+\int_s^Th_n(r,X^{n}_r,Y^{n}_r)\,{\rm d}B_r-\int_s^TZ^{n}_r\,{\rm d}W_r,\ 0\leq s \leq T.
\label{eqmarkov1bis}
\end{eqnarray}
where 
\begin{eqnarray*}
f_{n}(r,x,y)={\bf 1}_{[t_{n},T]}f(r,x,y),\,g_{n}(r,x,y)={\bf 1}_{[t_{n},T]}g(r,x,y),\; h_{n}(r,x,y)={\bf 1}_{[t_{n},T]}h(r,x,y).
\end{eqnarray*}
We consider the extension $X^n_s=x_n,\; Y^n_s=Y^n_{t_n},\; A^n_s=U^n_s=V^n_s=Z^n_s=0$ if $s\in [0,\,t_n]$.
The first part of the proof study the $S$-tightness of the process $(X^n,A^n,Y^n, M^n, K^{1,n},K^{2,n})$ where for all $ s\in [0,\,T]$,
\begin{eqnarray*}
M^{n}_s=\int_{t_n}^{s\vee t_n}Z^n_rdW_r=\int_{t_n}^{s\vee t_n}\widehat{Z}^{t_n,x_n}_rdM_r^{X^n},\; K^{1,n}_s=\int_{t_n}^{s\vee t_n}U^n_rdr,\;\;\; K^{2,n}=\int_{t_n}^{s\vee t_n}V^{n}_rdA_r^{n},
\end{eqnarray*}
with $\displaystyle dM_r^{X^n}=\sigma(X^n_r)dW_r$. With the notation below, equation \eqref{eqmarkov1bis} is written
\begin{eqnarray*}
&&Y^{n}_s+(K^{1,n}_T-K^{1,n}_s)+(K^{2,n}_T-K^{2,n}_s)\nonumber\\&=&\chi(X^{n}_T)+\int_s^Tf_n(r,X^{n}_r,Y^{n}_r)\,{\rm
d}r+\int_s^Tg_n(r,X^{n}_r,Y^{n}_r)\,{\rm d}A^{n}_r\nonumber\\
&&+\int_s^Th_n(r,X^{n}_r,Y^{n}_r)\,{\rm d}B_r-(M^n_T-M^n_s),\ 0\leq s \leq T.
\end{eqnarray*}
It follows from the identical computation used in Step 1 and Step 2 of Theorem \ref{thm3.1} that
\begin{eqnarray}
&&\sup_{n\in\N}\E\left[\sup_{0\leq t\leq T}|Y^n_t|^2+\int_0^T|Y^n_s|^2A^n_s+\int_0^T|Z^n_s|^2ds\right.\nonumber\\
&&\left.+\int_0^T|U^n_s|^2ds+\int_0^T|V^n_s|^2dA^n_s\right]<+\infty.\label{bound}
\end{eqnarray}
and then there exists a constant independent of $n$ such that
\begin{eqnarray}
CV_T(Y^n)+CV_T(K^{1,n})+CV_T(K^{2,n})+CV_T(M^n)\leq C,\label{CV}
\end{eqnarray}
where for a càdlàg stochastic process $L$ such that $\E(|L_t|)<+\infty$,
\begin{eqnarray*}
CV_T(L)=\sup_{\pi}\sum_{i=0}^{N}\E\left[|\E(L_{t_{i+1}}-L_{t_{i}}|\mathcal{F}_{t_i})|\right],
\end{eqnarray*}
such that the supremum is taken over all partition $\pi: 0=t_0<t_1<\cdot\cdot\cdot<t_N=T$.  Next according to Theorem 16 in \cite{MR1}, it follows from \eqref{CV} that the sequence $(Y^n,M^n,K^{1,n},K^{2,n})$ is tight with respect to the $S$-topology. Therefore, there exists a subsequence still denoted by $(M^n,K^{1,n},K^{2,n})$ and the process $(\bar{Y}^n,\bar{M}^n,\bar{K}^{1,n},\bar{K}^{2,n})$ in $\mathcal{D}([0,T],\R)^4$ such that $(Y^n,M^n,K^{1,n},K^{2,n})$ and $(\bar{Y}^n,\bar{M}^n,\bar{K}^{1,n},\bar{K}^{2,n})$ are equal in law and $(\bar{Y}^n,\bar{M}^n,\bar{K}^{1,n},\bar{K}^{2,n})\rightarrow(\bar{Y},\bar{M},\bar{K}^1,\bar{K}^2)$ in sense of  $S$-topology. Finally adapted arguments used in the proof of Proposition 4.3 in \cite{KLA} or Lemma 5, Lemma 6 and Lemma 7 in \cite{MR1}, we  obtain
\begin{eqnarray*}
\bar{Y}=Y^{t,x},\;\bar{M}=M^{t,x},\;\bar{K}^1=K^{1,t,x},\;\bar{K}^2=K^{2,t,x}.
\end{eqnarray*}
In particular, since $Y^n_{t_n}$ and $\bar{Y}^n_{t_n}$ are $\mathcal{F}_{t_n}$-measurable, they can be regarded as a random variable defined on $\Omega_2$ and $Y^n_{t_n}=\bar{Y}^n_{t_n}$. Therefore since for the same reason as below $Y^{t,x}_t$ is a $\mathcal{F}_{t,T}$-measurable random defined on $\Omega_2,\; Y^{t_n,x_n}_{t_n}\rightarrow Y^{t,x}_t, \; \P_2$-a.s.
\end{proof}
Let us define
\begin{eqnarray}
u(t,x)=Y^{t,x}_t.
\end{eqnarray}
Then $u$ is random field such that $u(t,x)$ is $\mathcal{F}^{B}_{t,T}$-measurable for each $(t,x)\in[0,T]\times\overline{\Theta}$.

We are now ready to derive the main result of this section.
\begin{theorem}
Assume $(\bf H6)$-$(\bf H9)$ hold. Then,
the function $u$ defined above is a stochastic viscosity
solution of variational SPDE \eqref{SPVI}.
\end{theorem}
\begin{proof}
First, since $u(t,x)=Y^{t,x}_t$, it follows from Proposition \ref{Prop} that $u\in C({\bf F}^B,\;[0,T]\times \overline{\Theta})$ and $u(T,x)=\chi(x)$. On the other hand, it follows from \eqref{eqmarkov1} that for all $(\tau,\xi)\in\mathcal{M}_{0,T}^B\times L^2\left(\mathcal{F}^{B}_{\tau,T};\overline{\Theta}\right)$\newline $u(\tau,\xi)\in {\rm Dom}(\varphi)$ on $\left\{\xi\in\Theta\right\}$ and $u(\tau,\xi)\in {\rm Dom}(\psi)$ on $\left\{\xi\in Bd(\Theta)\right\}$.

Thus it remains to show \eqref{E:def1}-\eqref{E:viscosity01} and \eqref{E:def01} -\eqref{E:viscosity001}. In other word, using
Corollary \ref{corollary4.8}, it suffices to prove that $v(t, x) =\varepsilon(t, x, u(t, x))$ satisfies \eqref{E:def2}-\eqref{E:viscosity02} and \eqref{E:def02}-\eqref{E:viscosity002}. In this fact, we are going to use the Yosida approximation of \eqref{eqmarkov1},
which was studied in Section 2.  For each $(t,x)\in[0,T]\times\overline{\Theta},\; \delta>0$, let $\{(Y^{t,x,\delta}_{s},Z^{t,x,\delta}_{s}),\,\ 0\leq s\leq T\}$ denote the solution of the following GBDSDE:
\begin{eqnarray}
&&Y^{t,x,\delta}_s+\int_s^T\nabla\varphi_{\delta}(Y^{t,x,\delta}_r)\,{\rm
d}r
+\int_s^T\nabla\psi_{\delta}(Y^{t,x,\delta}_r)\,{\rm
d}A_r\nonumber\\
&=&\chi(X^{t,x}_T)+\int_s^Tf(r,X^{t,x}_r,Y^{t,x,\delta}_r)\,{\rm
d}r+\int_s^Tg(r,X^{t,x}_r,Y^{t,x,\delta}_r)\,{\rm
d}A_r\nonumber\\
&&+\int_s^T h(r,X^{t,x}_r,Y^{t,x,\delta}_r) \,{\rm
d}B_r-\int_s^TZ^{t,x,\delta}_r\,{\rm d}W_r,\ t\leq s \leq
T.\label{eqmarkovapp}
\end{eqnarray}

Define $Y^{t,x,\delta}_t=u^{\delta}(t,x)$,  it is well known (see
Theorem 4.7, \cite{Bal}) that the function $v^{\delta}(t, x) = \varepsilon(t,
x, u^{\delta}(t, x))$ is an $\omega_2$-wise viscosity solution to the
following SPDE with nonlinear Dirichlet-Neumann boundary condition
\begin{eqnarray}
\left\{
\begin{array}{l}
{\rm(i)}\;\displaystyle\left(\frac{\partial v^{\delta}}{\partial
t}(t,x)-\left[
Lv^{\delta}(t,x)+\widetilde{f}_{\delta}(t,x,v^{\delta}(t,x))\right]\right)=0,\,\,\
(t,x)\in[0,T]\times\Theta,\\\\
{\rm(ii)}\;\displaystyle\frac{\partial v^{\delta}}{\partial
n}(t,x)+\widetilde{g}_{\delta}(t,x,v^{\delta}(t,x))=0,\,\,\ (t,x)\in[0,T]\times\partial\Theta, \\\\
{\rm(iii)}\; v(T,x)=\chi(x),\,\,\,\,\,\,\ x\in\overline{\Theta},
\end{array}\right.\label{SPDE}
\end{eqnarray}
where
$$\widetilde{f}_{\delta}(t,x,y,z)=\widetilde{f}(t,x,y)-\frac{\nabla\varphi_{\delta}(\eta(t,x,y))}{D_y\eta(t,x,y)}\;\;\;\mbox{and}
\;\;\;\widetilde{g}_{\delta}(t,x,y)=\widetilde{g}(t,x,y)-\frac{\nabla\psi_{\delta}(\eta(t,x,y))}{D_y\eta(t,x,y)}.$$
for all $(\tau,\xi)\in \in\mathcal{M}_{0,T}^B\times L^2\left(\mathcal{F}^{B}_{\tau,T};\overline{\Theta}\right)$, it follows from section 2 that (along a subsequence) $v^\delta(\tau,\xi)$ converge to $v(\tau,\xi)$ almost surely as $\delta $ goes to 0. Let $\omega_2\in\Omega_2$ be fixed such
\begin{eqnarray*}
|v^{\delta}(\tau(\omega_2),\xi(\omega_2))-v(\tau(\omega_2),\xi(\omega_2))|\rightarrow 0\;\;\; \mbox{as}\;\;\; \delta\rightarrow 0,
\end{eqnarray*}
and consider $(a_v,p_v,X_v)\in\mathcal{J}^{1,2,+}_{0}(v(\tau(\omega_2),\xi(\omega_2)))$. Then, it follows from Crandall-
Ishii-Lions \cite{CIL} that there exist sequences
\begin{eqnarray*}
\left\{
\begin{array}{ll}
\delta_n(\omega_2)\searrow 0,\\\\
(\tau_n(\omega_2),\xi_n(\omega_2))\in[0,T]\times\overline{\Theta},\\\\
(a_v^{n},p_v^{n},X_v^{n})\in\mathcal{J}^{1,2,+}_{0}(v^{\delta_n}(\tau_n(\omega_2),\xi_n(\omega_2)))
\end{array}
\right.
\end{eqnarray*}
such that
\begin{eqnarray*}
(\tau_n(\omega_2),\xi_n(\omega_2),a_v^{n},p_v^{n},X_v^{n},v^{\delta_n}(\tau_n(\omega_2),
\xi_n(\omega_2)))\rightarrow
(\tau(\omega_2),\xi(\omega_2),a_v,p_v,X_v,v(\tau(\omega_2),\xi(\omega_2))), \
\mbox{as} \ n\to \infty.
\end{eqnarray*}
Since $v^{\delta_n}(\omega_2,\cdot,\cdot)$ is a (deterministic) viscosity
solution to the PDE
$(\widetilde{f}_{\delta_n}(\omega_2,\cdot,\cdot,\cdot),0,\widetilde{g}_{\delta_n}(\cdot,\cdot),\chi)$,
we obtain
\begin{itemize}
\item [(a)] $(\tau_n(\omega_2),\xi_n(\omega_2))\in [0,T]\times\Theta$
\begin{eqnarray*}
&&V_{\widetilde{f}_{\delta_n}(\omega_2)}(\tau_n(\omega_2),\xi_n(\omega_2),a^n_v,X^n_v,p^n_v)+\frac{\nabla\varphi_{\delta_n}(\eta(\tau_n(\omega_2),\xi_n(\omega_2),v^{\delta_n}
(\tau_n(\omega_2),\xi_n(\omega_2))))}
{D_y\eta(\tau_n(\omega_2),\xi_n(\omega_2),v^{\delta_n}(\tau_n(\omega_2),\xi_n(\omega_2)))}\nonumber\\
&&\leq 0,\label{V01}
\end{eqnarray*}
\item [(b)]$(\tau_n(\omega_2),\xi_n(\omega_2))\in [0,T]\times\partial\Theta$
\begin{eqnarray*}
&&
\min\left(
\begin{array}{ll}V_{\widetilde{f}_{\delta_n}(\omega_2)}(\tau_n(\omega_2),\xi_n(\omega_2),a^n_v,X^n_v,p^n_v)+\frac{\nabla\varphi_{\delta_n}(\eta(\tau_n(\omega_2),\xi_n(\omega_2),v^{\delta_n(\tau_n}(\omega_2),\xi_n(\omega_2))))}{D_y\eta(\tau_n(\omega_2),\xi_n(\omega_2),v^{\delta_n}
(\tau_n(\omega_2),\xi_n(\omega_2)))},&\\\\
\langle\nabla\phi(\xi_n),\;p^n_v\rangle-\widetilde{g}_{\delta_n}(\omega_2)(\tau_n(\omega_2),\xi_n(\omega_2),v^{\delta_n}(\tau_n(\omega_2),\xi_n(\omega_2)))\\\\
+\frac{\nabla\psi_{\delta^n}(\eta(\tau_n(\omega_2),\xi_n(\omega_2),v^{\delta_n}(\tau_n(\omega_2),\xi_n(\omega))))}
{D_y\eta(\tau_n(\omega_2),\xi_n(\omega_2),v^{\delta_n}(\tau_n(\omega_2),\xi_n(\omega_2)))}&
\end{array}
\right)\nonumber\\
&&\leq
0\label{V02}.
\end{eqnarray*}
\end{itemize}
 In the sequel and for simplicity, we will omit writing $\omega_2$. Let take $y\in {\rm Dom}(\varphi)\cap {\rm Dom}(\psi)$ such
that $y\leq u(\tau,\xi)=\eta(\tau,\xi,v(\tau,\xi))$. Since $v^{\delta_n}$ converges uniformly in probability to $v$, there exists
$n_0 > 0$ such that  $y<\eta(\tau_n,\xi_n,v^{\delta_n}(\tau_n,\xi_n))$ for all $n\geq n_0$. Therefore, inequality \eqref{V01} and \eqref{V02} imply

\begin{eqnarray*}
&&(\eta(\tau_n,\xi_n,v^{\delta_n}(\tau_n,\xi_n))-y)V_{\widetilde{f}}(\tau_n,\xi_n,a_v^n,X_v^n,p_v^n)\\
&&\leq
\left(-\varphi_{\delta_n}(J_{\delta_n}(\eta(\tau_n,\xi_n,v^{\delta_n}(\tau_n,\xi_n))))+\varphi(y)\right)\frac{1}
{D_y\eta(\tau_n,\xi_n,v^{\delta_n}(\tau_n,\xi_n))},
\end{eqnarray*}
 and
\begin{eqnarray*}
\min\left(
\begin{array}{ll}
\eta(\tau_n,\xi_n,v^{\delta_n}(\tau_n,\xi_n))-y)V_{\widetilde{f}}(\tau_n,\xi_n,a_v^n,X_v^n,p_v^n)
+\frac{\varphi_{\delta_n}(J_{\delta_n}(\eta(\tau_n,
\xi_n,v^{\delta_n}(\tau_n,\xi_n))))-\varphi(y)}{D_y\eta(\tau_n,\xi_n,v^{\delta_n}
(\tau_n,\xi_n))},&\\\\
\left(\eta(\tau_n,\xi_n,v^{\delta_n}(\tau_n,\xi_n))-y\right)\left
[\langle\nabla\phi(\xi_n),p^n_v)\rangle
-\widetilde{g}(\tau_n,\xi_n,v^{\delta_n}(\tau_n,\xi_n))\right]&\\\\
+\frac{\psi_{\delta_n}(\bar{J}_{\delta_n}(\eta(\tau_n,
\xi_n,v^{\delta_n}(\tau_n,\xi_n))))-\psi(y)}{D_y\eta(\tau_n,\xi_n,v^{\delta_n}
(\tau_n,\xi_n))}&
\end{array}
\right)
\leq 0.
\end{eqnarray*}
Passing in the limit in the two below inequality, we get for all
$y\leq\eta(\tau,\xi,v(\tau,\xi))$, 
\begin{eqnarray*}
(V_{\widetilde{f}}(\tau,\xi,a_v,X_v,p_v)\leq -\frac{\varphi(\eta(\tau,\xi,v(\tau,\xi)))-\varphi(y)}{\eta(\tau,\xi,v(\tau,\xi))-y)}\frac{1}
{D_y\eta(\tau,\xi,v(\tau,\xi))},
\end{eqnarray*}
and
\begin{eqnarray*}
\min\left(
\begin{array}{ll}
V_{\widetilde{f}}(\tau,\xi,a_v,X_v,p_v)+\frac{\varphi(\eta(\tau,
\xi,v(\tau,\xi)))-\varphi(y)}{(\eta(\tau,\xi,v(\tau,\xi))-y)D_y\eta(\tau,\xi,v(\tau,\xi))},&\\\\
\left[\langle\nabla\phi(\xi), p_v)\rangle
-\widetilde{g}(\tau,\xi,v(\tau,\xi))\right]&\\\\
+\frac{\psi(\eta(\tau(\omega),
\xi,v(\tau,\xi)))-\psi(y)}{(\eta(\tau,\xi,v(\tau,\xi))-y)D_y\eta(\tau,\xi,v
(\tau,\xi))}&
\end{array}
\right)\leq 0.\label{V2}
\end{eqnarray*}

Taking the limit when $y$ goes to $\eta(\tau,\xi,v(\tau,\xi))$  we have

\begin{eqnarray*}
&&(V_{\widetilde{f}}(\tau,\xi,a_v,X_v,p_v)+\frac{\varphi'_l(\eta(\tau,\xi,v(\tau,\xi)))}
{D_y\eta(\tau,\xi,v(\tau,\xi))}\leq 0,\label{V1}
\end{eqnarray*}
and

\begin{eqnarray*}
\min\left(
\begin{array}{ll}
V_{\widetilde{f}}(\tau,\xi,a_v,X_v,p_v)+\frac{\varphi'_l(\eta(\tau,
\xi,v(\tau,\xi)))}{D_y\eta(\tau,\xi,v(\tau,\xi))}, &\\\\
\langle\nabla\phi(\xi),D_x\eta(\tau,\xi,v(\tau,\xi)p_v)\rangle
-\widetilde{g}(\tau,\xi,v(\tau,\xi))+\frac{\psi'_l(\eta(\tau,
\xi,v(\tau,\xi)))}{D_y\eta(\tau,\xi,v (\tau,\xi))}&
\end{array}
\right)\leq
0.
\end{eqnarray*}
which implies that $v$ satisfies \eqref{E:def2} and
\eqref{E:viscosity02}. Then, it follows from Corollary
\ref{corollary4.8} that $u$ is a stochastic viscosity subsolution of MSPDE \eqref{SPVI}. By
similar arguments, one can prove that $u$ is a stochastic viscosity
supersolution of MSPDE \eqref{SPVI} and completes the proof.
\end{proof}
\begin{remark}
In this work, we established existence and uniqueness result for variational generalized backward doubly stochastic differential equations under non-Lipchitz condition. Then using this result, we only gave a probabilistic representation of the viscosity solution of the variational stochastic partial differential equations with nonlinear Neumann condition. Since the study of uniqueness requires additional conditions, we promise to discuss this in a complementary article.
\end{remark}

{\bf Acknowledgements}

The work was partial done while the first author was visiting Shandong University and was finalized when the second author stayed at Anhui Normal University.
The first author would like to thank Prof. Shige Peng for providing a stimulating working environment. Next, the second one thank Prof. Yong Ren and all the administration of the Department of Mathematics for their support and  hospitality during his stay in March 2018.

\label{lastpage-01}

\begin{thebibliography}{99}



\bibitem{KLA} Bahlaly K, Maticiuc L,  and Z$\breve{a}$linescu A. Penalization method for a nonlinear Neumann PDE via weak solutions of reflected SDEs. Electron. J. Probab., 2018, 18 (102): 1-19.


\bibitem{Bal} Boufoussi B, Van castern J and Mrhardy N. Generalized backward doubly stochastic differential equations and SPDEs with nonlinear Neumann boundary conditions.
 Bernoulli, 2007, 13 (2): 423-446

\bibitem{BM} Boufoussi B and Mrhardy N. Multivalued stochastic
partial differential equations via backward doubly stochastic
differential equations. Stoch. Dyn., 2008, 8(2): 271-294

\bibitem{Br} Brezis H. Op\'{e}ateurs maximaux monotones.
Mathematics studies: North Holland, 1973

\bibitem{BMa} Buckdahn R and Ma J. Pathwise stochastic Taylor expansions and stochastic viscosity solutions for fully nonlinear stochastic PDEs. Ann. Probab., 2002, 30 (3): 1131-1171

\bibitem {BM1} Buckdahn R, and J. Ma J. Stochastic viscosity solutions for nonlinear stochastic partial differential equations (Part II). Stochastic Process. Appl., 2001, 93 (2) 205-228

\bibitem {BM2} Buckdahn R. and Ma J. Stochastic viscosity solutions for nonlinear stochastic partial differential equations (Part I). Stochastic Process. Appl.,  2001, 93 (2): 181-204

\bibitem{CIL} Crandall M G, Ishii H and Lions P L,  User's guide to the viscosity solutions of second order partial differential equations. Bull. Amer. Math. Soc., 1992, 27 (1): 1-67
 




\bibitem{Geg} Gegout-Petit A and Pardoux E. Equations différentielles stochastiques rétrogrades reflèchies dans un convexe. Stochastics Stochastics Rep., 1996, 57 (1-2): 111-128




\bibitem{LZ} Lions P L and Sznitman A S.
Stochastic differential equation with reflecting boundary
conditions, Comm. Pure Appl. Math., 1984, 37 (4): 511-537

\bibitem{LS1} Lions P L and Souganidis P E. Fully nonlinear stochastic partial differential equations. C.R. Acad. Sci. Paris, 1998, 326 (1): 1085-1092

\bibitem{LS2} Lions P L and Souganidis P E. Fully nonlinear stochastic partial differential equations: non-smooth equations and applications. C. R. Acad. Sci. Paris, 1998, 327 (1): 735-741
\bibitem{XM} Mao X. Adapted solution of backward differential equations with non-Lipschitz coefficients. Stochastic Process. Appl., 1995,  58 (2): 281-202

\bibitem{MR1} Maticiuc L and R\u{a}\c{s}canu A. On the continuity of the probabilistic representation of a semilinear Neumann-Dirichlet problem. Stochastic process. Appli., 2016, 126 (2): 572-607

\bibitem{MR} Maticiuc L and R\u{a}\c{s}canu A. A stochastic approach to a multivalued Dirichlet-Neumann problem. Stochastic Process. Appl., 2010, 120 (6): 777-800

\bibitem{ON} N'zi M and Owo J M. Backward doubly stochastic differential equations with non-Lipschitz coefficients. Random Operators/Stochastic Eqs., 2008, 16 (4): 307-324


\bibitem{PP}Pardoux E. and Peng S. Backward doubly stochastic differential equations and systems of quasilinear SPDE. Probab. Theory Related Fields, 1994, 88: (2) 209-227

\bibitem{PR} Pardoux E and R\u{a}\c{s}canu A. Backward stochastic differential equations with sub-differential operator and related variational inequalities, Stochastic Process. Appl., 1998, 76 (2): 191-215
\bibitem{PR1} Pardoux E and R\u{a}\c{s}canu E. Continuity of the Feynman-Kac formula for a generalized parabolic equation, Stochastics, 2017,  89 (5): 726-752
\bibitem{PZ}  Pardoux E and Zhang S. Generalized BSDEs and nonlinear Neumann boundary value problems, Probab. Theory Related Fields, 1998,  110: 535-558


\bibitem{Pr} Protter P E. Stochastic Integration and Differential Equations, 2nd ed., Stochastic Modeling and Applied Probability, Version 2.1, Springer, Berlin, 2005



\bibitem{Shial} Yufeng Shi , Yanling Gu and Kai Liu, Comparison Theorems of Backward Doubly Stochastic Differential Equations and Applications, Stochastic Analysis and Applications, 2005, 23 (1): 97-110 
 
 \bibitem{WH} Y. Wang and Z. Huang, Backward Stochastic Differential Equations with non-Lipschitz coefficients. Preprint, 2008.
 
 \bibitem{WW} Y. Wang and X. Wang, Adapted solutions of backward SDE with non-Lipschitz coefficients (in chinese). Chinese J. Probab. Statist., 2003, 19: 245-251.

\end{thebibliography}
\end{document}